\documentclass[11pt]{amsart}
\usepackage[OT2,T1]{fontenc}
\newcommand\textcyr[1]{{\fontencoding{OT2}\fontfamily{wncyr}\selectfont #1}}
\usepackage[a4paper,lmargin=29.5mm,rmargin=29.5mm,bmargin=29.5mm,tmargin=29.5mm]{geometry}
\usepackage{amsmath,amssymb,amsthm,amsfonts,mathrsfs,mathtools,relsize,stmaryrd,mathscinet,wasysym}
\usepackage[numbers,sort&compress,longnamesfirst]{natbib}
\usepackage[usenames,dvipsnames,svgnames,table]{xcolor}
\usepackage{hyperref}
\hypersetup{
	colorlinks,
	linkcolor={black},
	citecolor={black},
	urlcolor={blue!60!black},
	pdftitle={Three-dimensional graph products with unbounded stack-number},
	pdfauthor={Eppstein--Hickingbotham--Merker--Seweryn--Norin--Wood},
	colorlinks=true,
	filecolor={blue!80!black}
}

\newcommand{\defn}[1]{\textcolor{Maroon}{\emph{#1}}}

\usepackage[noabbrev,capitalise]{cleveref}
\crefname{lem}{Lemma}{Lemmas}
\crefname{thm}{Theorem}{Theorems}
\crefname{cor}{Corollary}{Corollaries}
\crefname{prop}{Proposition}{Propositions}
\crefname{conj}{Conjecture}{Conjectures}

\crefformat{equation}{(#2#1#3)}
\Crefformat{equation}{Equation #2(#1)#3}

\renewcommand{\le}{\leqslant}
\renewcommand{\leq}{\leqslant}
\renewcommand{\ge}{\geqslant}
\renewcommand{\geq}{\geqslant}

\newtheorem{thm}{Theorem}
\newtheorem{conjecture}[thm]{Conjecture}

\newtheorem{lem}[thm]{Lemma}
\newtheorem{cor}[thm]{Corollary}

\newtheorem{prop}[thm]{Proposition}

\newcommand{\mc}[1]{\mathcal{#1}}

\newcommand{\bb}[1]{\mathbb{#1}}

\newcommand{\brm}[1]{\operatorname{#1}}

\newcommand{\fs}[2]{\left(\frac{#1}{#2}\right)}
\newcommand{\s}[1]{\left(#1\right)}

\newcommand{\had}[2]{h_{#1}(#2)}
\newcommand{\Jfour}{J_{\Diamond}}
\newcommand{\Jsix}{J_{\protect\resizebox{2mm}{!}{\hexagon}}}

\DeclareMathOperator{\sn}{sn}
\DeclareMathOperator{\dsn}{dsn}
\DeclareMathOperator{\qn}{qn}
\DeclareMathOperator{\congestion}{cong}
\DeclareMathOperator{\Tr}{Tr}
\DeclareMathOperator{\overlap}{overlap}

\title[Three-dimensional graph products with unbounded stack-number]{Three-dimensional graph products\\ with unbounded stack-number}
\author[Eppstein]{David Eppstein}
\address{Department of Computer Science, University of California, Irvine, California, USA}
\email{eppstein@uci.edu}
\author[Hickingbotham]{Robert Hickingbotham}
\address{School of Mathematics, Monash University, Melbourne, Australia}
\email{robert.hickingbotham@monash.edu}
\author[Merker]{Laura Merker}
\address{Institute of Theoretical Informatics, Karlsruhe Institute of Technology, Germany}
\email{laura.merker2@kit.edu}
\author[Norin]{Sergey Norin}
\address{Department of Mathematics and Statistics, McGill University, Montreal, QC, Canada.}
\email{snorin@math.mcgill.ca}
\author[Seweryn]{Micha\l{} T. Seweryn}
\address{Department of Theoretical Computer Science, Jagiellonian University, Kraków, Poland}
\email{michal.seweryn@tcs.uj.edu.pl}
\author[Wood]{David R. Wood}
\address{School of Mathematics, Monash University, Melbourne, Australia}
\email{david.wood@monash.edu}

\thanks{Norin is supported by an NSERC Discovery Grant. Hickingbotham is supported by an Australian Government Research Training Program Scholarship. Wood is supported by the Australian Research Council.}
\begin{document}

\begin{abstract}
	We prove that the stack-number of the strong product of three $n$-vertex paths is $\Theta(n^{1/3})$. The best previously known upper bound was $O(n)$. No non-trivial lower bound was known. This is the first explicit example of a graph family with bounded maximum degree and unbounded stack-number. 
	
	The main tool used in our proof of the lower bound is the topological overlap theorem of Gromov. We actually prove a stronger result in terms of so-called triangulations of Cartesian products. We conclude that triangulations of three-dimensional Cartesian products of any sufficiently large connected graphs have large stack-number.
	
	The upper bound is a special case of a more general construction based on families of permutations derived from Hadamard matrices. 
	
	The strong product of three paths is also the first example of a bounded degree graph with bounded queue-number and unbounded stack-number. A natural question that follows from our result is to determine the smallest $\Delta_0$ such that there exist a graph family with unbounded stack-number, bounded queue-number and maximum degree $\Delta_0$. We show that $\Delta_0\in \{6,7\}$.
\end{abstract}



\maketitle

\vspace*{-3ex}

\section{Introduction}

Stack layouts are ubiquitous objects at the intersection of combinatorics, geometry and topology with applications in computational complexity~\cite{GKS89,DSW16,Bourgain09,BY13}, 
RNA folding~\cite{HS99}, 
graph drawing in two~\cite{BB04,SSSV-JGT96} 
and three dimensions~\cite{Wood-GD01}, traffic light scheduling~\cite{Kainen90}, and fault-tolerant multiprocessing~\cite{CLR87,Rosenberg83a}. 

For a graph\footnote{All graphs in this paper are simple and, unless explicitly stated otherwise, undirected and finite. Let $\mathbb{N}=\{1,2,\dots\}$ and $\mathbb{N}_0=\mathbb{N}\cup\{0\}$.} $G$ and $s\in\mathbb{N}_0$, an \defn{$s$-stack layout} of $G$ consists of an ordering $(v_1,\ldots,v_n)$ of $V(G)$ together with a function $\phi\colon E(G) \to \{1,\dots,s\}$ such that for all edges $v_iv_j,v_kv_\ell \in E(G)$ with $i < k < j < \ell$ we have $\phi(v_iv_j) \neq \phi(v_k v_\ell)$; see~\cref{StackLayout} for an example. Each set $\phi^{-1}(k)$ is called a \defn{stack}. Edges in a stack do not cross with respect to $(v_1,\ldots,v_n)$, and therefore behave in a LIFO manner (hence the name). Stack layouts can also be viewed topologically via embeddings into books (first defined by Ollmann~\cite{Ollmann73}). The \defn{stack-number $\sn(G)$} of a graph $G$ is the minimum $s\in\mathbb{N}_0$ for which there exists an $s$-stack layout of $G$ (also known as \defn{page-number}, \defn{book-thickness} or \defn{fixed outer-thickness}). 

\begin{figure}[!h]
\centering
\includegraphics{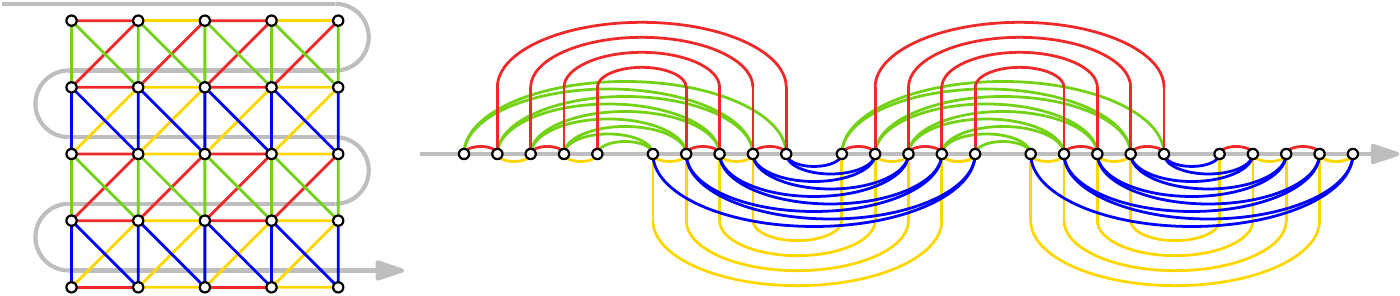}
\caption{A \(4\)-stack layout of the strong product \(P_5 \boxtimes P_5\) of two paths.}
\label{StackLayout}
\end{figure}

Stack layouts have been studied for 
planar graphs~\citep{BS84,BKKPRU20,Yannakakis89,Yann20,Heath-FOCS84}, 
graphs of given genus~\citep{Malitz94b,HI92,Endo97}, 
treewidth~\citep{GH01,DW07,VWY09,DujWoo11}, 
minor-closed graph classes~\citep{Blankenship-PhD03}, 
1-planar graphs~\citep{BBKR17,B20,BLGGMR20,ABK15}, and 
graphs with a given number of edges~\citep{Malitz94a}, 
amongst other examples. 


This paper studies stack layouts of 3-dimensional products. As illustrated in \cref{Grids}, for graphs $ G_1 $ and $ G_2 $, the \defn{Cartesian product $G_1\boxempty G_2$} is the graph with vertex-set $ V(G_1) \times V(G_2) $ with an edge between two vertices $ (x,y) $ and $ (x',y') $ if $x=x'$ and $ yy' \in E(G_2) $, or $y=y'$ and $xx'\in E(G_1) $. The \defn{strong product \(G_1 \boxtimes G_2\)} is the graph obtained from $G_1\boxempty G_2$ by adding edges $(x,y)(x',y')$ and $(x,y')(x',y)$ for all edges $xx' \in E(G_1)$ and $yy' \in E(G_2)$. Since the Cartesian and strong products are associative, we may write $G_1\boxempty G_2\boxempty G_3$ and $G_1\boxtimes G_2\boxtimes G_3$ (identifying pairs of the forms \(((v_1, v_2), v_3)\) and \((v_1, (v_2, v_3))\) with the triple \((v_1, v_2, v_3)\)).	
\begin{figure}[b]
	\centering
	\includegraphics{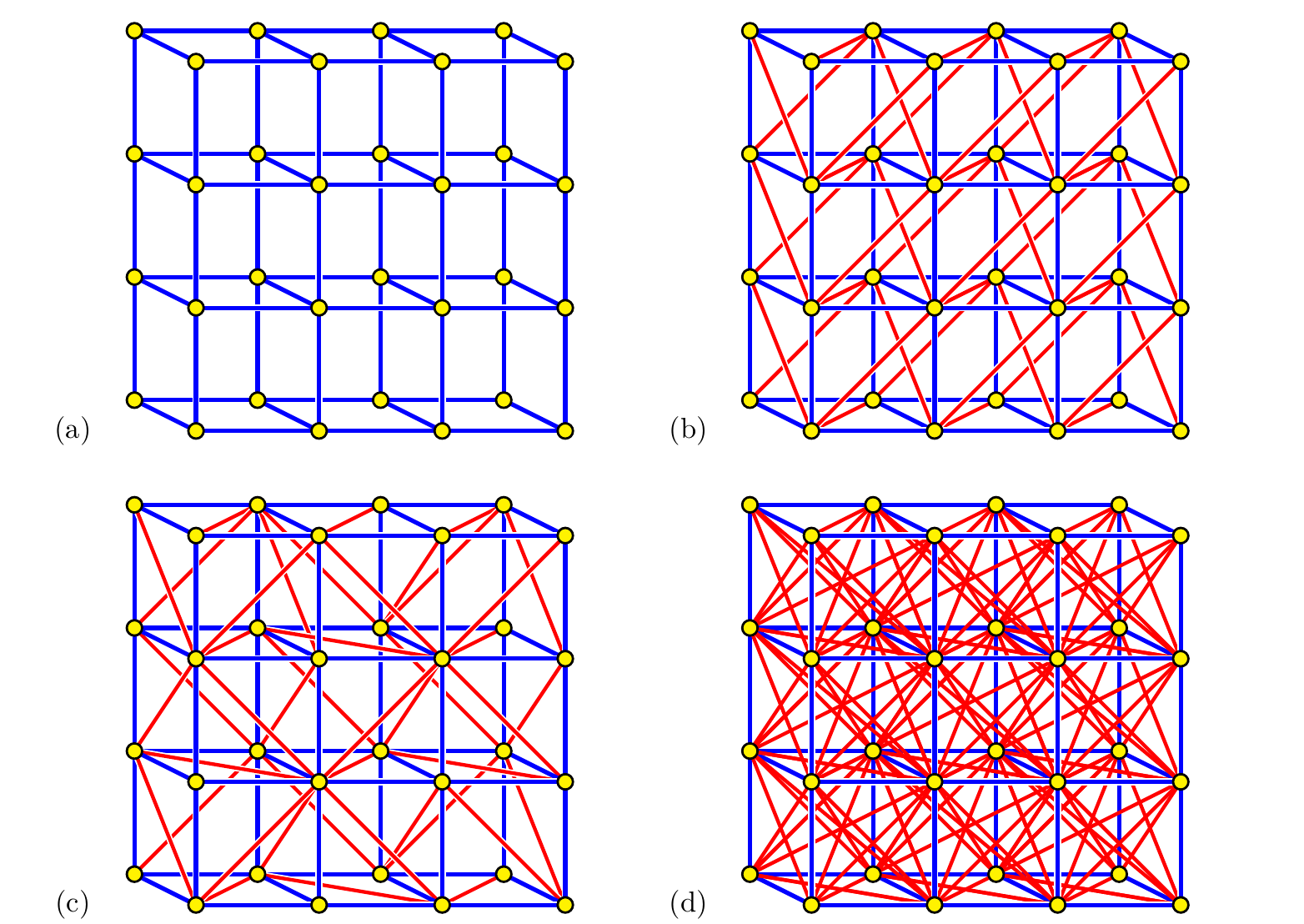}
	\caption{(a) $P_4\boxempty P_4\boxempty P_2$, (b) $P_4\boxslash P_4\boxslash P_2$, (c) a triangulation of $P_4\boxempty P_4\boxempty P_2$, (d) $P_4\boxtimes P_4\boxtimes P_2$.}
	\label{Grids}
\end{figure}

Let \defn{$P_n$} denote the $n$-vertex path. Our first main result is the following tight bound on the stack-number of the strong product of three paths (the 3-dimensional grid plus crosses).

\begin{thm}
\label{t:StrongProductPaths}
$\sn(P_n \boxtimes P_n \boxtimes P_n) \in \Theta(n^{1/3})$. 
\end{thm}

Note that \((P_n \boxtimes P_n) \boxempty P_n\) and \((P_n \boxempty P_n) \boxtimes P_n\) both have bounded stack-number, as we now explain. \citet{BK79} showed that if \(G_1\) and \(G_2\) are graphs with bounded stack-number and \(G_1\) is bipartite with bounded maximum degree, then the stack-number of \(G_1 \boxempty G_2\) is bounded. \citet{Pupyrev20a} showed that if additionally \(G_2\) has bounded pathwidth, then the stack-number of \(G_1 \boxtimes G_2\) is also bounded. These results imply that \((P_n \boxtimes P_n) \boxempty P_n\) and \((P_n \boxempty P_n) \boxtimes P_n\) indeed have bounded stack-number. This shows that in \cref{t:StrongProductPaths}, we cannot replace even one `strong product' by a `Cartesian product'.

No non-trivial lower bound on $\sn(P_n\boxtimes P_n\boxtimes P_n)$ was previously known. Indeed, \cref{t:StrongProductPaths} provides the first explicit example of a graph family with bounded maximum degree and unbounded stack-number. \citet{Malitz94a} first proved that graphs of maximum degree $3$ have unbounded stack-number (using a probabilistic argument). Further motivation for \cref{t:StrongProductPaths} is provided in \cref{QueueConnections} where we present various connections to related graph parameters, shallow/small minors and growth. 

We now discuss the lower bound in \cref{t:StrongProductPaths}. We actually prove a stronger result that depends on the following definitions. For graphs $G_1$ and $G_2$, a \defn{triangulation} of $G_1\boxempty G_2$ is any graph obtained from $G_1\boxempty G_2$ by adding the edge $(x,y)(x',y')$ or $(x,y')(x',y)$ for each $xx'\in E(G_1)$ and $yy'\in E(G_2)$. 
 A \defn{triangulation} of $G_1 \boxempty G_2 \boxempty G_3$ is any graph obtained by triangulating all subgraphs induced by sets of the form $\{v_1\} \times V(G_2) \times V(G_3)$, $V(G_1) \times \{v_2\} \times V(G_3)$ and $V(G_1)\times V(G_2) \times \{v_3\}$ with $v_i \in V(G_i)$; see \cref{Grids}(c) for an example.
 
 For directed graphs $G_1$ and $G_2$, 
 if $U_i$ is the undirected graph underlying $G_i$, then 
 let \defn{$G_1\boxslash G_2$} be the triangulation of $U_1 \boxempty U_2$ containing the edge $(x,y)(x',y')$ for all directed edges $(x,x')\in E(G_1)$ and $(y,y')\in E(G_2)$.
 Similarly, for directed graphs \(G_1\), \(G_2\) and \(G_3\), let \defn{$G_1\boxslash G_2\boxslash G_3$} be the appropriate triangulation of \(U_1\boxempty U_2\boxempty U_3\). When using the notation \(G_1 \boxslash G_2\) or \(G_1 \boxslash G_2 \boxslash G_3\), if some $G_i$ is a path \(P_n\), then \(P_n\) is a directed path. 

Every triangulation of $G_1\boxempty G_2$ is a subgraph of $G_1\boxtimes G_2$ and every triangulation of $G_1\boxempty G_2 \boxempty G_3$ is a subgraph of $G_1\boxtimes G_2 \boxtimes G_3$. So the next result immediately implies the lower bound in \cref{t:StrongProductPaths}.

\begin{thm}
\label{t:TreeTriangulation}
Let $T_1$, $T_2$ and $T_3$ be $n$-vertex trees with maximum degree $\Delta_1$, $\Delta_2$ and $\Delta_3$ respectively, where $\Delta_1\geq\Delta_2\geq\Delta_3$. Then for every triangulation $G$ of $T_1\boxempty T_2\boxempty T_3$, 
$$\sn(G) \in \Omega\left(\left(
	\frac{n}{\Delta_1\Delta_2}
	\right)^{1/3}
	\right).$$ 
\end{thm}

\cref{t:TreeTriangulation} is similar in spirit to a recent result of 
\citet{DEHMW21}, who showed that if \defn{$S_n$} is the $n$-leaf star, then \[\sn(S_n\boxempty(P_n\boxslash P_n))\in\Omega(\sqrt{\log\log n}).\]
Their proof actually establishes the following result (since the Hex Lemma holds for any triangulation of $P_n\boxempty P_n$; see \citep{Wood18} for example). 

\begin{thm}[\citep{DEHMW21}] 
	\label{StarPathPath}
	For every triangulation $G$ of $P_n\boxempty P_n$,
	\[\sn(S_n\boxempty G)\in\Omega(\sqrt{\log\log n}).\]
\end{thm}

\cref{t:TreeTriangulation} has the advantage over \cref{StarPathPath} in that it applies to bounded degree graphs (for example when each $T_i$ is a path). Moreover, the lower bound in \cref{t:StrongProductPaths} (as a function of the number of vertices) is much stronger than the lower bound of \cref{StarPathPath}.

We prove \cref{t:TreeTriangulation} by relating the stack-number of a graph to the topological properties of its triangle complex.
The \defn{triangle complex} of a graph $G$, denoted by $\Tr(G)$, is the geometric $2$-dimensional simplicial complex with $0$-faces corresponding to vertices of $G$, $1$-faces corresponding to edges of $G$, and $2$-faces corresponding to triangles of $G$. 
For topological spaces $X$ and $Y$, define 
$$\overlap(X,Y) = \min_{f \in C(X,Y)} \max_{p \in Y} | f^{-1}(p)|,$$ 
where \defn{$C(X,Y)$} denotes the space of all continuous functions $f: X \to Y$. 
In \cref{s:lem1,s:lem2}, respectively, we prove the following two lemmas.

\begin{lem}\label{l:PPPOverlap2} 
For all $n$-vertex trees $T_1$, $T_2$ and $T_3$, and for every triangulation $G$ of $T_1 \boxempty T_2 \boxempty T_3$, 
	$$\overlap(\Tr(G), \bb{R}^2) \in \Omega (n).$$
\end{lem}	

\begin{lem}\label{l:SNvsOverlap}
For every graph $G$ such that every vertex is in at most $c$ triangles,
$$\sn(G) \geq \s{\frac{\overlap(\Tr(G),\bb{R}^2)}{3c}}^{1/3}.$$ 
\end{lem}

\cref{t:TreeTriangulation} (and thus the lower bound in \cref{t:StrongProductPaths}) follows from \cref{l:PPPOverlap2,l:SNvsOverlap} since 
(using the notation from \cref{t:TreeTriangulation}) each vertex of $G$ is in at most $2(\Delta_1\Delta_2 + \Delta_1\Delta_3 + \Delta_2\Delta_3)\leq 6\Delta_1\Delta_2$ triangles.
So \cref{l:SNvsOverlap} is applicable with $c=6\Delta_1\Delta_2$. 

The lower bound on $\sn(G)$ in \cref{t:TreeTriangulation} is non-trivial only if $\Delta_1\Delta_2\in o(n)$. Nevertheless, we have the following result with no assumption on the maximum degree. \cref{t:TreeTriangulation,StarPathPath} imply that the stack-numbers of triangulations of \(P_n \boxempty P_n \boxempty P_n\) and \(S_n \boxempty P_n \boxempty P_n\) grow with \(n\). Moreover, \(S_n \boxempty S_n\) contains a subgraph isomorphic to a \(1\)-subdivision of \(K_{n, n}\), so its stack-number grows with \(n\) as well (see \citep{Eppstein01,BO99}). Since every sufficiently large connected graph contains a copy of \(P_n\) or \(S_n\), we deduce the following.

\begin{cor}\label{t:GraphTriangulation}
For every \(s\in\mathbb{N}\), there exists \(n\in\mathbb{N}\) such that for all $n$-vertex connected graphs $G_1$, $G_2$ and $G_3$, if $G$ is any triangulation of $G_1\boxempty G_2\boxempty G_3$, then $\sn(G) > s$.
\end{cor}


The best previously known upper bound on $\sn(P_n\boxtimes P_n\boxtimes P_n)$ was $O(n)$, which follows from Theorem~1 of \citet{DMY21} or from Corollary~1 of \citet{Pupyrev20a}. The upper bound in \cref{t:StrongProductPaths} follows from a more general result in \cref{UpperBoundSection}, which says that $\sn(G \boxtimes P_n)\in O(n^{1/2-\epsilon})$ for some $\epsilon>0$ whenever the graph $G$ has bounded stack-number and bounded maximum degree. The proof is based on families of permutations derived from Hadamard matrices.

Our final results concern maximum degree. \cref{t:TreeTriangulation} implies that $(P_n \boxslash P_n \boxslash P_n)_{n\in\mathbb{N}}$ is a family of graphs with maximum degree 12, unbounded stack-number and bounded queue-number (defined in \cref{QueueConnections}). It is natural to ask what is the smallest bound on the maximum degree in such a family. We prove the answer is 6 or 7. 

\begin{thm}\label{t:Maxsnqn}
The least integer \(\Delta_0\) such that there exists a graph family with maximum degree \(\Delta_0\), unbounded stack-number and bounded queue-number satisfies \(\Delta_0 \in \{6, 7\}\).
\end{thm}

The proof of the upper bound in \cref{t:Maxsnqn} uses the same topological machinery used to prove \cref{t:TreeTriangulation}, and is based on the tesselation of $\mathbb{R}^3$ by truncated octahedra. The proof of the lower bound exploits a connection with clustered colouring.

\section{Connections}
\label{QueueConnections}

This section provides further motivation for our results by discussing connections with related graph parameters, minors and growth.

Consider a graph $G$. The \defn{geometric thickness} of $G$ is the minimum $k\in\mathbb{N}_0$ for which there is a straight-line drawing of $G$ and a partition of $E(G)$ into $k$ plane subgraphs; see \citep{BMW06,DEH00,DW07}. A $k$-stack layout of $G$ defines such a drawing and edge-partition (with the vertices drawn on a circle in the order given by the stack layout). Thus the geometric thickness of $G$ is at most its stack-number. The \defn{slope-number} of $G$ is the minimum $k\in\mathbb{N}_0$ for which there is a straight-line drawing of $G$ with $k$ distinct edge slopes; see \citep{BMW06,PP06,DESW07,DSW07,ABH06,KPPT08}. Since edges of the same slope do not cross, the geometric thickness of $G$ is at most its slope-number. 

Note that $P_n \boxslash P_n \boxslash P_n$ has slope-number and geometric thickness at most 6 (simply project the natural 3-dimensional representation to the plane). Hence \cref{t:StrongProductPaths} provides a family of graphs with bounded slope-number, bounded geometric thickness, and unbounded stack-number. \citet{Eppstein01} previously constructed a graph family with bounded geometric thickness and unbounded stack-number, but not with bounded slope-number (since the graphs in question have unbounded maximum degree). 

For a graph $G$ and $q\in\mathbb{N}_0$, a \defn{$q$-queue layout} of $G$ consists of an ordering $(v_1,\ldots,v_n)$ of $V(G)$ together with a function $\phi: E(G) \to \{1,\dots,q\}$ such that for all edges $v_iv_j,v_kv_\ell \in E(G)$ with $i < k< \ell <j$ we have $\phi(v_iv_j) \neq \phi(v_kv_\ell)$. Each set $\phi^{-1}(k)$ is called a \defn{queue}. Edges in a queue do not nest with respect to $(v_1,\ldots,v_n)$, and therefore behave in a FIFO manner (hence the name). The \defn{queue-number $\qn(G)$} of a graph $G$ is the minimum $q\in\mathbb{N}_0$ for which there exists a $q$-queue layout of $G$. 

Stack and queue layouts are considered to be dual to each other~\citep{HLR92}. However, for many years it was open whether there is a graph family with bounded queue-number and unbounded stack-number, or bounded stack-number and unbounded queue-number. \cref{StarPathPath} by \citet{DEHMW21} resolved the first question, since they also showed that $\qn(S_n\boxempty (P_n\boxslash P_n))\leq 4$. Observe that a \(4\)-queue layout of \(P_n \boxslash P_n \boxslash P_n\) can be obtained by taking the lexicographical ordering \((v_1, \ldots, v_{n^3})\) of the vertices and letting \(\phi(v_i v_j)\) be determined by which of the sets \(\{1\}\), \(\{n, n+1\}\), \(\{n^2, n^2 + 1\}\), \(\{n^2 + n\}\) contains \(|i - j|\). By \cref{t:StrongProductPaths}, $(P_n \boxslash P_n \boxslash P_n)_{n\in\mathbb{N}}$ is therefore another graph family with queue-number at most 4 and unbounded stack-number. 

We now discuss the behaviour of stack- and queue-number with respect to taking minors. Let $G$, $H$ and $J$ be graphs. Let $r,s\in\mathbb{N}_0$ and let $k\geq 0$ be a half-integer (that is, $2k\in\mathbb{N}_0$). $H$ is a \defn{minor} of $G$ if a graph isomorphic to $H$ can be obtained from $G$ by vertex deletions, edge deletions, and edge contractions. A \defn{model} of $H$ in $G$ is a function $\mu$ with domain $V(H)$ such that: $\mu(v)$ is a connected subgraph of $G$; $\mu(v)\cap \mu(w)=\emptyset$ for all distinct $v,w\in V(G)$; and $\mu(v)$ and $\mu(w)$ are adjacent for every edge $vw \in E(H)$. Observe that $H$ is a minor of $G$ if and only if $G$ contains a model of $H$. For $r\in\mathbb{N}_0$, if there exists a model $\mu$ of $H$ in $G$ such that $\mu(v)$ has radius at most $r$ for all $v \in V(H)$, then $H$ is an \defn{$r$-shallow minor} of $G$. For $s\in\mathbb{N}_0$, if there exists a model $\mu$ of $H$ in $G$ such that $|V(\mu(v))|\leq s$ for all $v \in V(H)$, then $H$ is an \defn{$s$-small minor} of $G$. We say that $J$ is an \defn{$(\leq s)$-subdivision} of $H$ if $J$ can be obtained from $H$ by replacing each edge by a path with at most $s$ internal vertices. If every path that replaces an edge has exactly $s$ vertices, then $J$ is the \defn{$s$-subdivision} of $H$. We say that $H$ is a \defn{$k$-shallow topological minor} of $G$ if a subgraph of $G$ is isomorphic to a $(\leq 2k)$-subdivision of $H$. Note that if a graph $H$ is an $s$-small minor or an $s$-shallow topological minor of a graph $G$, then $H$ is an $r$-shallow minor of $G$ whenever $s\leq r$.

\citet{BO99} conjectured that stack-number is `well-behaved' under shallow topological minors in the following sense:

\begin{conjecture}[\cite{BO99}]
\label{B_conj}
There exists a function $f$ such that for every graph $G$ and half-integer $k\geq 0$, if $H$ is any $k$-shallow topological-minor of $G$, then $\sn(H) \leq f(\sn(G),k)$.
\end{conjecture}

\citet{DEHMW21} disproved \cref{B_conj}. Their proof used the following lemma by \citet{DujWoo05}.

\begin{lem}[\citep{DujWoo05}]
\label{QueueSubdiv}
For every graph $G$, if \(s = 1+2 \lceil \log_2\qn(G)\rceil\) then the $s$-subdivision of \(G\) has a 3-stack layout. 
\end{lem}

\cref{QueueSubdiv} implies that the $5$-subdivision of $S_n\boxempty(P_n\boxslash P_n)$ admits a $3$-stack layout. Using \cref{StarPathPath}, \citet{DEHMW21} concluded there  exists a graph class $\mc{G}$ with bounded stack-number for which the class of 1-shallow topological minors of graphs in $\mc{G}$ has unbounded stack-number. Thus stack-number is not well-behaved under shallow topological minors. We now prove an analogous result for small minors. 

\begin{thm}
\label{StackSmallMinorBad}
There exists a graph class $\mc{G}$ with bounded stack-number for which the class of 2-small minors of graphs in $\mc{G}$ has unbounded stack-number. 
\end{thm}	

\begin{proof}
\cref{QueueSubdiv} implies that the $5$-subdivision of $P_n \boxslash P_n \boxslash P_n$ admits a $3$-stack layout. $P_n \boxslash P_n \boxslash P_n$ is a $37$-small minor of its $5$-subdivision since $P_n \boxslash P_n \boxslash P_n$ has maximum degree at most $12$ and \(12 \lceil \frac{5}{2}\rceil+1=37\). Let \(\mathcal{G}_0\) be the class of graphs with stack-number at most \(3\), and for each \(i \in\mathbb{N}_0 \), let \(\mathcal{G}_{i+1}\) be the class of \(2\)-small minors of graphs in \(\mathcal{G}_i\). Thus, \(\mathcal{G}_{37}\) contains all \(37\)-small minors of graphs in \(\mathcal{G}_0\), including all graphs of the form \(P_n \boxslash P_n \boxslash P_n\). Hence, there exists \(i\in\{0,\dots,37\}\) such that stack-number is bounded for  \(\mathcal{G}_i\) and unbounded for \(\mathcal{G}_{i+1}\). 
\end{proof}

In contrast to \cref{StackSmallMinorBad}, queue-number is well-behaved under small minors. In fact, the following lemma shows it is even well-behaved under shallow minors, which is a key distinction between these parameters.

\begin{lem}[\citep{HW21b}] 
	\label{ShallowMinorQueue}
	For every graph $G$ and for every $r$-shallow minor $H$ of $G$,
	$$\qn(H) \leq 2r(2 \qn(G))^{2r}.$$
\end{lem}

We now compare stack and queue layouts with respect to growth. 
The \defn{growth} of a graph $G$ is the function $f_G\colon\mathbb{N}\to \mathbb{N}$ such that \(f_G(r)\) is the maximum number of vertices in a subgraph of $G$ with radius at most $r$. Similarly, the \defn{growth} of a graph class \(\mc{G}\) is the function $f_{\mc{G}}\colon\mathbb{N}\to \mathbb{N}\cup\{\infty\}$ where \(f_{\mc{G}}(r)=\sup\{f_G(r)\colon G\in\mc{G}\}\). A graph class $\mathcal{G}$ has \defn{linear}/\defn{quadratic}/\defn{cubic}/\defn{polynomial growth} if $f_{\mathcal{G}}(r)\in O(r)/O(r^2)/O(r^3)/O(r^d)$ for some $d\in \mathbb{N}$. Let $\mathbb{Z}_n^d$ be the $d$-fold strong product $P_n\boxtimes\dots\boxtimes P_n$. Every subgraph in $\mathbb{Z}^d_n$ with radius at most $r$ has at most $(2r+1)^d$ vertices. Thus $(P_n\boxtimes P_n\boxtimes P_n)_{n\in \mathbb{N}}$ has cubic growth, so \cref{t:StrongProductPaths} implies that graph classes with cubic growth can have unbounded stack-number. In contrast, graphs with polynomial growth have bounded queue-number. \citet{KL07} established the following characterisation of graphs with polynomial growth.

\begin{thm}[\cite{KL07}]\label{t:Growth}
	If a graph $G$ has growth $f_G(r)\in O(r^d)$,	then $G$ is isomorphic to a subgraph of $\mathbb{Z}^{O(d\log d)}_n$ for some $n\in\mathbb{N}$. 
\end{thm}

It follows from the upper bound on the queue-number of products by \citet{Wood05} that $\qn(\mathbb{Z}^d_n)\leq c\,3^d$ for some constant $c$. \cref{t:Growth} therefore implies the following.

\begin{cor}\label{c:GrowthQueue}
	If a graph $G$ has growth $f_G(r)\in O(r^d)$, then $\qn(G)\in 2^{O(d\log d)}$. 
\end{cor}

 \section{Proof of \cref{l:PPPOverlap2}}\label{s:lem1}

The topological arguments in this paper exclusively involve finite polyhedral $2$-dimensional cell complexes, and so for brevity we refer to a finite polyhedral $2$-dimensional cell complex as simply a \defn{complex}\footnote{See \cite[Section ~12]{Bjorner95} for background on finite polyhedral cell complexes, which are referred to there as convex linear cell complexes.}. The proof of \cref{l:PPPOverlap2} relies on a Topological Overlap Theorem of Gromov~\cite{Gromov10}. To use it we need a technical variation of the overlap parameter. For a complex $X$ and topological space $Y$, define $$\overlap_{\triangle}(X,Y) = 
\min_{f \in C(X,Y)} \max_{p \in Y} |\{F \in X^{=2} \;|\; p \in f(F)\}|,$$
where $X^{=2}$ denotes the set of $2$-dimensional cells of $X$. We now state the $2$-dimensional case of Gromov's theorem.

\begin{thm}[{\cite[p.~419, Topological $\Delta$-inequality]{Gromov10}}]\label{t:Gromov} There exists $\alpha>0$ such that
	$$\overlap_{\triangle}(\Tr(K_n),\bb{R}^2) \geq \alpha n^3$$ 
for every integer $n \geq 3$.
\end{thm}

We use \cref{t:Gromov} in combination with the following straightforward lemma to lower bound $\overlap(X,\bb{R}^2)$ for a complex $X$ in terms of $\overlap_{\triangle}(\Tr(K_n),X)$. 

\begin{lem}\label{l:factor}
For every complex $X_0$ and for all topological spaces $X$ and $Y$, 
	\[
	  \overlap(X,Y) \geq \frac{\overlap_{\triangle}(X_0,Y)}{\overlap_{\triangle}(X_0,X)}.
	\]
\end{lem}
\begin{proof}
	Let $f_0\colon X_0 \to X$ be a continuous function such that 
	$$\overlap_{\triangle}(X_0,X) = \max_{p \in X} |\{F \in X_0^{=2} | p \in f_0(F)\}|.$$
	Let $f\colon X \to Y$ be an arbitrary continuous function. Then $f_0 \circ f\colon X_0 \to Y$ is continuous and so there exists $p_0 \in Y$ such that $$|\{F \in X_0^{=2} \:|\: p_0 \in (f_0 \circ f)(F)\}| \geq \overlap_{\triangle}(X_0,Y).$$
	But 
	\begin{align*}|\{F \in X_0^{=2} \:|\: p_0 \in (f_0 \circ f)(F)\}| &\leq \sum_{p \in f^{-1}(p_0)}|\{F \in X_0^{=2} | p \in f_0(F)\}| \\&\leq | f^{-1}(p_0)| \cdot \overlap_{\triangle}(X_0,X).
	\end{align*}
	It follows that $$| f^{-1}(p_0)| \geq \frac{\overlap_{\triangle}(X_0,Y)}{\overlap_{\triangle}(X_0,X)},$$
	as desired. 
\end{proof}	

A family $\mc{B}$ of subcomplexes of a complex $X$ is a \defn{bramble over $X$}\footnote{This definition is inspired by the definition of a \emph{bramble} in a graph, which is used in graph minor theory. A bramble of a graph is only required to satisfy the first and second among the conditions we impose.} if: 
\begin{itemize}
	\item every $B \in \mc{B}$ is non-empty,
	\item $B_1 \cup B_2$ is connected for every pair of distinct $B_1,B_2 \in \mc{B}$,
	\item $B_1 \cup B_2 \cup B_3$ is simply connected for every triple of distinct $B_1,B_2, B_3 \in \mc{B}$.
\end{itemize}	

The \defn{congestion $\congestion(\mc{B})$} of a bramble $\mc{B}$ is the maximum size of a collection of elements in $\mc{B}$ that all share a point in common; that is, $$\congestion(\mc{B}) = \max_{p \in X} \:|\: \{B \in \mc{B} \:|\: p \in B\}|.$$ 

\begin{lem}\label{l:bramble} There exists $\beta>0$ such that for every complex $X$ and bramble $\mc{B}$ over $X$, 
	$$\overlap(X,\bb{R}^2) \geq \beta \frac{|\mc{B}|}{\congestion(\mc{B})}.$$
\end{lem}	

\begin{proof}
	Let $\mc{B}=\{B_1,\ldots,B_n\}$. Let $X_0= \Tr(K_n)$ be the triangle complex of the complete graph with vertex set $\{v_1,\ldots,v_n\}$. We construct a continuous map $f: X_0 \to X$ as follows. Let $p_i \in B_i$ be chosen arbitrarily, and set $f(v_i)=p_i$ for every $i \in \{1,\dots,n\}$. Extend $f$ to the $1$-skeleton of $X_0$ by mapping each edge $v_i v_j$ to a path $\pi_{ij} \subseteq B_i \cup B_j$ from $p_i$ to $p_j$. 
	
	Finally, we need to continuously extend $f$ to the interior of every $2$-simplex $F_{ijk}$ of $X_0$ with vertices $v_i,v_j$ and $v_k$. Since $B_i \cup B_j \cup B_k$ is simply connected, and the boundary of $F_{ijk}$ is mapped to a closed curve in $B_i \cup B_j \cup B_k$, there is such an extension such that $f(F_{ijk}) \subseteq B_i \cup B_j \cup B_k$.
	
	Consider arbitrary $p \in X$, and let $I = \{i \in \{1,\dots,n\} \:|\: p \in B_i\}$. Then $|I| \leq \congestion(\mc{B})$ and $$|\{F \in X^{=2} \:|\: p \in f(F)\}| \leq |\{\{i,j,k\} \subseteq \{1,\dots,n\} \:| \: \{i,j,k\} \cap I \neq \emptyset \}| \leq 3|I|\binom{n}{2}.$$
	It follows that $\overlap_{\triangle}(X_0,X) \leq \frac{3}{2}\congestion(\mc{B})n^2.$
	Thus by \cref{t:Gromov} and \cref{l:factor}, 
	$$\overlap(X,\bb{R}^2) \geq \frac{\overlap_{\triangle}(X_0,\bb{R}^2)}{\overlap_{\triangle}(X_0,X)} \geq \frac{\alpha n^3}{\frac{3}{2}\congestion(\mc{B})n^2},$$ where $\alpha$ is from \cref{t:Gromov}.
	It follows that $\beta= \frac{2}{3}\alpha$ satisfies the lemma.
\end{proof}

\cref{l:bramble} allows one to lower bound $\overlap(X,\bb{R}^2)$ by exhibiting a bramble with size large in comparison to its congestion. To simplify the verification of the conditions in the definition of a bramble we use a simple consequence of van Kampen's theorem.\footnote{See \cite[Theorem 1.20]{Hatcher02} for the general statement of van Kampen's theorem, or \cite[Section 2.1.3]{Camarena} for an example of the application of the theorem in the setting generalizing \cref{l:vanKampen}.} 

\begin{lem}\label{l:vanKampen} 
	Let $B_1,B_2$ be subcomplexes of a complex $X$. If $B_1$ and $B_2$ are simply connected, and $B_1 \cap B_2$ is non-empty and connected, then $B_1 \cup B_2$ is simply connected. 
\end{lem}	

\begin{cor}\label{c:bramble} 
	Let $X$ be a complex. Let $\mc{B}$ be a family of subcomplexes of $X$ such that:
	\begin{itemize}
		\item every $B \in \mc{B}$ is simply connected, 
		\item $B_1 \cap B_2$ is connected for every pair of distinct $B_1,B_2 \in \mc{B}$, 
		\item $B_1 \cap B_2 \cap B_3$ is non-empty for every triple of distinct $B_1,B_2, B_3 \in \mc{B}$.
	\end{itemize}
Then $\mc{B}$ is a bramble over $X$.
\end{cor}	

\begin{proof}
	The first condition in the definition of a bramble trivially holds. 
	The second condition holds since $B_1 \cup B_2$ is simply connected for all distinct $B_1,B_2 \in \mc{B}$ by \cref{l:vanKampen}.
	
	It remains to show that $B_1 \cup B_2 \cup B_3$ is simply connected for all distinct $B_1,B_2, B_3 \in \mc{B}$. Since $B_1 \cup B_2$ and $B_3$ are simply connected, this follows from \cref{l:vanKampen} as long as $B_3 \cap (B_1 \cup B_2) = (B_3 \cap B_1) \cup (B_3 \cap B_2)$ is non-empty and connected. By the assumptions of the corollary, each of $B_3\cap B_1$ and $B_3 \cap B_2$ is connected and they share a point in common. Thus their union is connected as desired.
\end{proof}

We are now ready to prove \cref{l:PPPOverlap2}.

\begin{proof}[Proof of \cref{l:PPPOverlap2}.] 
Let $G$ be any triangulation of $T_1 \boxempty T_2 \boxempty T_3$, where $T_1$, $T_2$ and $T_3$ are $n$-vertex trees. To simplify our notation we assume that $T_1,T_2$ and $T_3$ are vertex-disjoint. We denote by $\widehat{T}_i$ the geometric $1$-dimensional complex corresponding to $T_i$, to avoid confusion between discrete and topological objects.
	
For $u \in V(T_1)$, let $G(u)$ be the subgraph of $G$ induced by all the vertices of $G$ with the first coordinate $u$. That is, $V(G(u))=\{(u,u_2,u_3) \: | \:u_2 \in V(T_2),u_3 \in V(T_3)\}$. So $G(u)$ is isomorphic to a triangulation of $T_2 \boxempty T_3$. Let $X(u) = \Tr(G(u))$. As a topological space, $X(u)$ is homeomorphic to $\widehat{T}_2 \times \widehat{T}_3$. Define $G(u)$ and $X(u)$ for $u \in V(T_2) \cup V(T_3)$ analogously. 
		
For $(u_1,u_2,u_3) \in V(G)$, let $B(u_1,u_2,u_3)=X(u_1) \cup X(u_2) \cup X(u_3)$, and let \linebreak $\mc{B}= \{B(u_1,u_2,u_3) \: | \: (u_1,u_2,u_3) \in V(G) \}$. 
We claim that $\mc{B}$ is a bramble over $\brm{Tr(G)}$. It suffices to check that it satisfies the conditions of \cref{c:bramble}.
	
To verify the first condition for $ \mc{B} $, we use \cref{c:bramble} to show that $\{X(u_1),X(u_2),X(u_3)\}$ is a bramble over $\brm{Tr(G)}$ for every $(u_1,u_2,u_3) \in V(G)$. Note that each $\widehat{T}_i$ is simply connected. As the product of simply connected spaces is simply connected, it follows that $X(u)$ is simply connected for every $u \in V(T_1) \cup V(T_2) \cup V(T_3)$. Consider now $(u_1,u_2,u_3) \in V(G)$. Then $X(u_1) \cap X(u_2)$ is homeomorphic to $\widehat{T}_3$ and is connected. Similarly, $X(u_1) \cap X(u_3)$ and $X(u_2) \cap X(u_3)$ are connected. Finally, $X(u_1) \cap X(u_2) \cap X(u_3)$ consists of a single point (and is connected). By \cref{c:bramble} the set $\{X(u_1),X(u_2),X(u_3)\}$ is a bramble. In particular, $B(u_1,u_2,u_3)=X(u_1) \cup X(u_2) \cup X(u_3)$ is simply connected. Thus the first condition in  \cref{c:bramble} for $ \mc{B} $ holds.

For the second condition, consider distinct $B(u_1,u_2,u_3), B(v_1,v_2,v_3) \in \mc{B}$. Let $R = B(u_1,u_2,u_3) \cap B(u_1,u_2,u_3)$ for brevity. Assume first, for simplicity, that $u_i \neq v_i$ for $i \in \{1,2,3 \}$. Then 
\[R = \bigcup_{\substack{i,j\in\{1,2,3\} \\ i \neq j}} X(u_i) \cap X(v_j).\]
Each set $X(u_i) \cap X(v_j)$ in this decomposition is connected, since it is homeomorphic to $\widehat{T}_k$ for $\{k\} = \{1,2,3\} \setminus \{i,j\}$. Moreover, when ordering these sets, 
\begin{align*}  
X(u_1) \cap X(v_2), \; & X(u_3) \cap X(v_2),\; X(u_3) \cap X(v_1),\\
& X(u_2) \cap X(v_1),\; X(u_2) \cap X(v_3),\; X(u_1) \cap X(v_3), 
\end{align*} 
each pair of consecutive sets share a point, for example, $(u_1,v_2,u_3) \in (X(u_1) \cap X(v_2)) \cap (X(u_3) \cap X(v_2)).$ It follows that their union is connected.

It remains to consider the case $u_i= v_i$ for some $i \in \{1,2,3 \}$. Assume $i=1$ without loss of generality. If $u_2 \neq v_2$ and $u_3 \neq v_3$ then $R = X(u_1) \cup (X(u_2) \cap X(v_3)) \cup (X(v_2) \cap X(u_3)).$ Each of the sets in this decomposition is again connected and the second and third sets intersect the first, implying that their union is connected. Finally, if say $u_2=v_2$ then $R= X(u_1) \cup X(u_2)$ and is again a union of connected intersecting sets. This verifies the second condition of  \cref{c:bramble} for $ \mc{B} $. 

For the last condition, for any $B(u_1,u_2,u_3), B(v_1,v_2,v_3),B(w_1,w_2,w_3) \in \mc{B}$,
$$ (u_1,v_2,w_3) \in B(u_1,u_2,u_3) \cap B(v_1,v_2,v_3) \cap B(w_1,w_2,w_3),$$
and so $B(u_1,u_2,u_3) \cap B(v_1,v_2,v_3) \cap B(w_1,w_2,w_3) \neq \emptyset$. 
By \cref{c:bramble}, $\mc{B}$ is a bramble.

As noted above, $|\mc{B}|=n^3$. Moreover, for every $p \in \Tr(G)$ and every $i \in \{1,2,3\}$ there exists at most one $v_i \in V(T_i)$ such that $p \in X(v_i)$. Thus $\congestion(\mc{B}) \leq 3n^2$, and by \cref{l:bramble},
	$$\overlap(\Tr(G), \bb{R}^2) \geq \beta \frac{|\mc{B}|}{\congestion(\mc{B})} \geq \frac{\beta}{3}n,$$
	as desired.
\end{proof}	

Note that $\Tr(P_n\boxtimes P_n \boxtimes P_n)$ has a natural affine embedding into $\bb{R}^3$ with vertices mapped to points in $\{1,\dots,n\}^{3}$. Since any line in a direction sufficiently close to an axis direction intersects this complex in $n+O(1)$ points, projecting along such a direction we obtain an affine map $\Tr(G) \to \bb{R}^2$ that  covers every point at most $n+O(1)$ times. Thus $\overlap(\Tr(P_n \boxtimes P_n  \boxtimes P_n), \bb{R}^2) \leq n+O(1)$ and so the bound in \cref{l:PPPOverlap2} can not be substantially improved.

\section{Proof of \cref{l:SNvsOverlap}}\label{s:lem2}

The proof of \cref{l:SNvsOverlap} depends on the following result.

\begin{lem}
	\label{CrossingsDisjointTriangles}
	Let $T_1,\dots,T_m$ be pairwise vertex-disjoint pairwise intersecting triangles in $\mathbb{R}^2$ with all the vertices on a circle $S$. Assume that the edges of $T_1,\dots,T_m$ can be partitioned into $k$ non-crossing sets. Then $m\leq k^3$.
\end{lem}

\begin{proof}
	Number the vertices of $T_1,\dots,T_m$ by $1,\dots,3m$ in clockwise order starting at an arbitrary point on $S$. By assumption there is a function $\phi\colon \bigcup_i E(T_i) \to \{1,\dots,k\}$ such that $\phi(e_1)\neq\phi(e_2)$ for all crossing edges $e_1,e_2\in\bigcup_i E(T_i)$. Say the vertices of $T_i$ are $(a_i,b_i,c_i)$ with $a_i<b_i<c_i$. Define the function $f \colon \{1,\dots,m\}\to\{1,\dots,k\}^3$ by $f(i)=(\phi(a_ib_i),\phi(a_ic_i),\phi(b_ic_i))$. Suppose that $f(i)=f(j)$ for distinct $i,j\in\{1,\dots,m\}$. Thus $a_ib_i$ does not cross $a_jb_j$, and $a_ic_i$ does not cross $a_jc_j$, and
	$b_ic_i$ does not cross $b_jc_j$. Without loss of generality, $a_i<a_j$. 
	If $c_i<c_j$ then $a_ic_i$ crosses $a_jc_j$, so $a_i<a_j<b_j<c_j<c_i$. 
	Now consider $b_i$.
	If $a_i<b_i<a_j$ or $c_j<b_i<c_i$, then $T_i$ and $T_j$ do not intersect.
	So $a_j<b_i<c_j$. 
	If $a_j<b_i<b_j$, then $a_ib_i$ crosses $a_jb_j$.
	Otherwise, $b_j<b_i<c_j$ implying $b_jc_j$ crosses $a_jc_j$. 
	This contradiction shows that 
	$f(i)\neq f(j)$ for distinct $i,j\in\{1,\dots,m\}$. Hence $m\leq k^3$. 
\end{proof}

\begin{proof}[Proof of \cref{l:SNvsOverlap}]
	Let $k=\sn(G)$, and let $(v_1,\ldots,v_n)$ together with a function $\phi \colon E(G) \to \{1,\dots,k\}$ be a $k$-stack layout of $G$. Let $p_1, p_2,\ldots, p_n$ be pairwise distinct points chosen on a circle $S$ in  $\bb{R}^2$, numbered in cyclic order around $S$. Let $X = \Tr(G)$. Define a continuous function $f\colon X \to \bb{R}^2$ by setting $f(v_i)=p_i$ for all $i \in \{1,\dots,n\}$, extending $f$ affinely to $2$-simplices of $X$ (triangles of $G$), and, for simplicity, mapping every edge of $G$ that does not belong to a triangle continuously to a curve internally disjoint from a circle bounded by $S$, so that every point of $\bb{R}^2$ belongs to at most two such curves.
	
	Let $m'= \overlap(X,\bb{R}^2)$. Then there exists $p \in \bb{R}^2$ such that $|f^{-1}(p)| \geq m'$. If $m' \leq 2$ the lemma trivially holds, and so we assume $m'>2$. Since the restriction of $f$ to each simplex of $X$ is injective, there exist triangles $T_1,T_2,\ldots, T_{m'} \subseteq \bb{R}^2$ corresponding to images of distinct $2$-simplices of $X$ so that $p \in \bigcap_{i=1}^{m'}T_i$. Let $m = \lceil \frac{m'}{3c-2}\rceil > \frac{m'}{3c}$. Since every triangle shares a vertex with at most $3c-3$ others, we may assume that $T_1,T_2,\ldots, T_{m}$ are pairwise vertex-disjoint by greedily selecting vertex-disjoint triangles among $T_1,T_2,\ldots, T_{m'}$. By \cref{CrossingsDisjointTriangles}, $m \leq k^3$. Thus
	\begin{equation*}
	\sn(G)=k \geq m^{1/3} \geq \fs{m'}{3c}^{1/3} = \s{\frac{\overlap(X,\bb{R}^2)}{3c}}^{1/3}.\qedhere 
	\end{equation*}
\end{proof}

This lemma completes the proof of \cref{t:StrongProductPaths,t:TreeTriangulation}.

\medskip
We finish this section by showing that \cref{CrossingsDisjointTriangles} is best possible up to a constant factor.

\begin{prop}
	For infinitely many $k\in\mathbb{N}$, there is a set of $(\frac{k}{3})^3$ pairwise intersecting and pairwise vertex-disjoint triangles with vertices on a circle in $\mathbb{R}^2$, such that the edges of the triangles can be $k$-coloured with crossing edges assigned distinct colours. 
\end{prop} 

This result is implied by the following (with $k=3\cdot2^\ell$).

\begin{lem}
	\label{InductiveTriangles}
	Let $S$ be a circle in $\mathbb{R}^2$ partitioned into three pairwise disjoint arcs $A,B,C$. For every $\ell\in \bb{N}_0$ there is a set $\mathcal{T}$ of $8^\ell$ triangles, each with one vertex in each of $A,B,C$, such that the $AB$-edges can be $2^\ell$-coloured with crossing edges assigned distinct colours, the $BC$-edges can be $2^\ell$-coloured with crossing edges assigned distinct colours, and the $CA$-edges can be $2^\ell$-coloured with crossing edges assigned distinct colours. 
\end{lem}

\begin{proof}
	We proceed by induction on $\ell$. The claim is trivial with $\ell=0$. Assume that for some integer $\ell\geq 0$, there is a set $\mathcal{T}$ of $8^\ell$ triangles satisfying the claim. 
	Let $X$ be the set of $2^\ell$ colours used for the $AB$-edges. 
	Let $Y$ be the set of $2^\ell$ colours used for the $BC$-edges. 
	Let $Z$ be the set of $2^\ell$ colours used for the $CA$-edges. 
	Let $X'=\{x',x'':x\in X\}$ be a set of $2^{\ell+1}$ colours. 
	Let $Y'=\{y',y'':y\in Y\}$ be a set of $2^{\ell+1}$ colours. 
	Let $Z'=\{z',z'':z\in Z\}$ be a set of $2^{\ell+1}$ colours. 
	Let $\mathcal{T}'$ be the set of triangles obtained by replacing each $T\in \mathcal{T}$ by 8 triangles as follows. 
	Say the vertices of $T$ are $u,v,w$ where $u\in A$, $v\in B$ and $w\in C$. 
	Replace $u$ by $u_1,\dots,u_8$ in clockwise order in $A$, 
	replace $v$ by $v_1,\dots,v_8$ in clockwise order in $B$, and 
	replace $w$ by $w_1,\dots,w_8$ in clockwise order in $C$. 
	By this we mean that if $p,q$ are consecutive vertices in the original ordering, then 
	$p_1,\dots,p_8,q_1,\dots,q_8$ are consecutive vertices in the enlarged ordering. 
	Add the triangles $u_1v_4w_6,u_2v_3w_5,u_3v_2w_8,u_4v_1w_7,u_5v_8w_2,u_6v_7w_1,u_7v_6w_4,u_8v_5w_3$ to $\mathcal{T}'$. 
	Say $uv$ is coloured $x\in X$, $vw$ is coloured $y\in Y$, and $wv$ is coloured $z\in Z$. 
	Colour each of $u_1v_4,u_2v_3,u_3v_2,u_4v_1$ by $x'\in X'$, and colour each of $u_5v_8,u_6v_7,u_7v_6,u_8v_5$ by $x''\in X'$. 
	Colour each of $v_1w_7,v_3w_5,v_5w_3,v_7w_1$ by $y'\in Y'$, and colour each of $v_2w_8,v_4w_6,v_6w_4,v_8w_2$ by $y''\in Y'$. 
	Colour each of $w_1u_6,w_2u_5,w_5u_2,w_6u_1$ by $z'\in Z'$, and colour each of $w_3u_8,w_4u_7, w_7u_4,w_8u_3$ by $z''\in Z'$. 
	As illustrated in \cref{ConstructTriangles}, crossing edges are assigned distinct colours. 
	Thus $\mathcal{T'}$ is the desired set of $8^{\ell+1}$ triangles.
\end{proof}

\begin{figure}[!h]
	\includegraphics{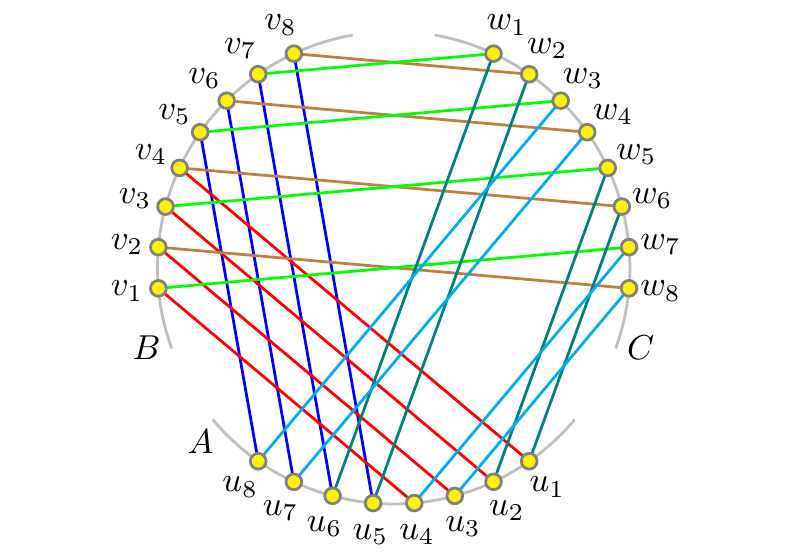}
	\caption{Construction in the proof of \cref{InductiveTriangles}. 
		\label{ConstructTriangles}}
\end{figure}

\section{Upper Bound}
\label{UpperBoundSection}

This section proves the upper bound in \cref{t:StrongProductPaths} showing that $P_n\boxtimes P_n\boxtimes P_n$ has a $O(n^{1/3})$-stack layout. Assume \(V(P_n) = \{1,\dots,n\}\) and \(E(P_n) = \{i(i+1): i \in \{1,\dots,n-1\}\}\).
We start with a sketch of the construction. 
Take a particular \(O(1)\)-stack layout and a proper
\(4\)-colouring of \(P_n \boxtimes P_n\).
In the corresponding ordering of \(V(P_n \boxtimes P_n)\)
replace each vertex \((x, y)\) by vertices
\(((x, y, z_1), \ldots, (x, y, z_n))\) where
\((z_1, \ldots, z_n)\) is a permutation of \(\{1,\dots,n\}\)
determined by the colour of \((x, y)\).
An appropriate choice of the permutations ensures
that the edges of \(P_n \boxtimes P_n \boxtimes P_n\)
can be partitioned into \(O(n^{1/3})\) stacks.
In particular, there can only be \(O(n^{1/3})\)
pairwise crossing edges with respect to the ordering,
so the length of a longest common subsequence of any
two of the permutations is \(O(n^{1/3})\). See \cref{Permutations} for an illustration of the case \(n=8\).

\begin{figure}[!h]
\includegraphics{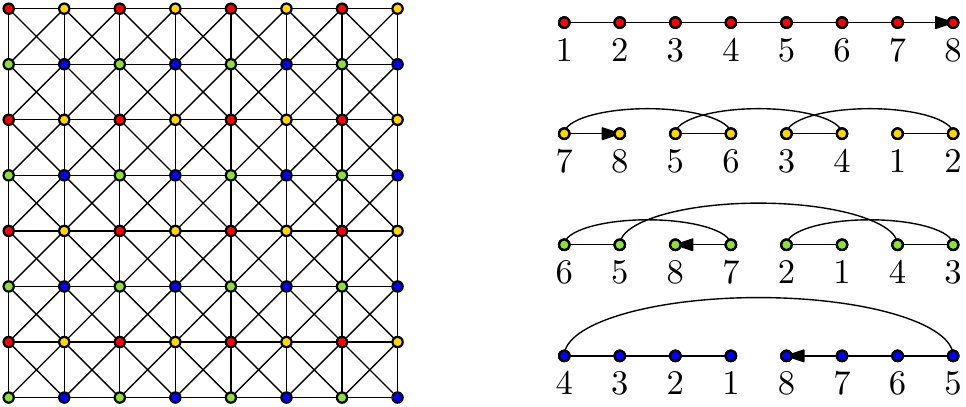}
\caption{Four permutations of \(\{1, \ldots, 8\}\), no two of which have a common subsequence of length greater than \(2\).}
\label{Permutations}
\end{figure}

We actually prove a more general result, \cref{thm:SnGboxtimesPn} below,
which relies on the following definition from the literature \citep{BK79,Pupyrev20a,DSW16,ABDGKP21,GKS89}. An $s$-stack layout $((v_1,\dots,v_n),\psi)$ of a graph $G$ is \defn{dispersable} (also called \defn{pushdown}) if $\psi^{-1}(k)$ is a matching in $G$ for each $k\in \{1,\dots,s\}$. The \defn{dispersable stack-number} $\dsn(G)$ is the minimum $s \in \mathbb{N}_0$ for which there exists a dispersable $s$-stack layout of $G$. For example, $\dsn(P_n\boxtimes P_n)\leq 8$ since in the 4-stack layout of \(P_n \boxtimes P_n\) illustrated in \cref{StackLayout}, each stack is a linear forest; putting alternative edges from each path in distinct stacks produces a dispersable 8-stack layout. In general, $\Delta(G)\leq \dsn(G)\leq(\Delta(G)+1)\sn(G)$ by Vizing's Theorem. So a graph family has bounded dispersable stack-number if and only if it has bounded stack-number and bounded maximum degree.

\begin{thm}\label{thm:SnGboxtimesPn}
  Let \(G\) be a graph with chromatic number \(\chi\)
  and dispersable stack-number \(d\). Let \(n \in \mathbb{N}\) and \(p = \max\{2^{\lceil\log_2 \chi\rceil}, 2\} \). Then 
  \[
    \sn(G \boxtimes P_n) \le 2^{p/2-1}d(2p-1) \cdot n^{1/2-1/(2p-2)} + 2p-3.
  \]
\end{thm}
Since \(\chi(P_n \boxtimes P_n) \le 4\) and \(\dsn(P_n \boxtimes P_n) \le 8\), \cref{thm:SnGboxtimesPn} implies that 
\[\sn(P_n \boxtimes P_n \boxtimes P_n) \le 2\cdot8\cdot7n^{1/3}+5=112n^{1/3}+5.\] 
Furthermore, \(\chi(G) \le \Delta(G) + 1 \le \dsn(G) + 1\). So
\cref{thm:SnGboxtimesPn} shows that graphs $G$ with bounded dispersable stack-number
satisfy \(\sn(G \boxtimes P_n)\in O(n^{1/2-\epsilon})\) for some \(\epsilon>0\).

For an even positive integer \(p\), a \defn{Hadamard matrix} of order \(p\) is a \(p \times p\) matrix \(H\) with all entries in \(\{+1, -1\}\) such that every pair of distinct rows differs in exactly \(p/2\) entries.
An example of a Hadamard matrix of order \(4\) is
\[
\begin{bmatrix}
+1 & +1 & +1 & +1\\
-1 & -1 & +1 & +1\\
-1 & +1 & -1 & +1\\
+1 & -1 & -1 & +1\\
\end{bmatrix}.
\]
\citet{Sylvester} proved that the order of a Hadamard matrix is \(2\) or is divisible by \(4\). He also constructed a Hadamard matrix of order any power of 2. \citet{Paley} constructed a Hadamard matrix of order $q+1$ for any prime power $q\equiv 3 \pmod 4$, and a Hadamard matrix of order $2(q+1)$ for any prime power $q\equiv 1 \pmod 4$. The Hadamard Conjecture proposes that there exists a Hadamard matrix of order \(p\) whenever \(p\) is divisible by 4. This conjecture has been verified for numerous values of $p$; see \citep{Hadamard} for example.

The following lemma captures a property of the construction of permutations without long common subsequences by \citet{BBHN} that is crucial in the proof of \cref{thm:SnGboxtimesPn}.

\begin{lem}\label{PermutationSums}
  Assume there exists a Hadamard matrix of order \(p\). Let \(m \in \mathbb{N}\) and \(n = m^{p-1}\). Then there exist permutations \(\pi_1, \ldots, \pi_p \colon \{1,\dots,n\} \to \{1,\dots,n\}\) such that for all distinct \(k, \ell \in \{1,\dots,p\}\), 
  \[
    |\{\pi_{k}(i) + \pi_{\ell}(j): i, j \in \{1,\dots,n\}, |i-j| \le 1\}| \le (2p-1)n^{1/2-1/(2p-2)},
  \]
  and for each \(k \in \{1,\dots,p\}\) there exists a \((2p-3)\)-stack layout of \(P_{n}\) using the vertex ordering \((\pi_k^{-1}(1), \ldots, \pi_k^{-1}(n))\).
\end{lem}

\begin{proof}
	Since there exists a Hadamard matrix of order \(p\), there exist \(\pm1\)-vectors \(h_{1}, \ldots, h_p\), any two of which differ in exactly \(p/2\) entries.
	Assume \(h_{k} = (\had{k}{0}, \ldots, \had{k}{p-1})\) for each $k\in\{1,\dots,p\}$. 	By possibly negating all entries in some of these vectors,
	we may assume \(\had{k}{p-1} = +1\) for all \(k \in \{1,\dots,p\}\).

  For each \(i \in \{1,\dots,n\}\), let \(d(i, 0), \ldots, d(i, p-2) \in \{0, \ldots, m-1\} \) denote the digits of \(i-1\) in base \(m\), so that
  \[
    i = 1+\sum_{a=0}^{p-2}d(i, a)m^a.
  \]
  Note that the mapping \(\{1,\dots,n\} \ni i \mapsto (d(i, 0), \ldots, d(i, p-2)) \in \{0, \ldots, m-1\}^{p-1}\) is bijective.

  For all \(k \in \{1,\dots,p\}\), \(i \in \{1,\dots,n\}\) and \(a \in \{0,\ldots,p-2\}\), let
  \[
  d_k(i, a) =
  \begin{cases}
  d(i, a)&\textrm{if \(\had{k}{a}=+1\),}\\
  m-1-d(i, a)&\textrm{if \(\had{k}{a}=-1\).}
  \end{cases}
  \]
  Observe that for any \(k, \ell \in \{1,\dots,p\}\) and \(a \in \{0, \ldots, p-2\}\), if $\had{k}{a}=-\had{\ell}{a}$ then $d_{k}(i, a)+d_{\ell}(i, a)=m-1$ for all $i \in \{1,\dots,n\}$. 

  For every \(k \in \{1,\dots,p\}\), define \(\pi_k\) by
  \[
    \pi_k(i) = 1+\sum_{a=0}^{p-2} d_k(i, a)m^a.
  \]

  We claim that \(\pi_1\), \ldots, \(\pi_p\) satisfy the lemma. Fix distinct \(k, \ell \in \{1,\dots,p\}\).  For \(i, j \in \{1,\dots,n\}\) and \(a \in \{0, \ldots, p-2\}\), define \(\tau_a(i, j) = d_{k}(i, a) + d_{\ell}(j, a)\), and let
  \[
    \tau(i, j) = (\tau_0(i, j), \ldots, \tau_{p-2}(i, j)).
  \] 
  Since \(\pi_{k}(i) + \pi_{\ell}(j) =
  \sum_{b=0}^{p-2} \tau_a(i, j)m^a\), the value of \(\tau(i, j)\) determines the value of \(\pi_{k}(i) + \pi_{\ell}(j)\). As such, to prove the first part of the lemma it suffices to show that
  \[
    |\{\tau(i, j): i, j \in \{1,\dots,n\}, |i-j|\le 1\}| \le (2p-1)n^{1/2-1/(2p-2)}.
  \]
  Let \(A^+ =\{a \in \{0, \ldots, p-2\}: \had{k}{a} = \had{\ell}{a}\}\) and 
  \(A^- =\{a \in \{0, \ldots, p-2\}: \had{k}{a} = -\had{\ell}{a}\}\). 
  Since \(h_{k}\) and \(h_{\ell}\) differ in exactly \(p/2\) entries and \(\had{k}{p-1} = +1 = \had{\ell}{p-1}\), we have \(|A^+| = p/2-1\)
  and \(|A^-| = p/2\).

  Let \(i \in \{1,\dots,n\}\).
  For every \(a \in A^-\), we have \(\tau_a(i, i)= m-1\) since \(\had{k}{a}=-\had{\ell}{a}\).
  For every \(a \in A^+\) we have \(d_{k}(i, a) = d_{\ell}(i, a) \in \{0, \ldots, m-1\}\), and thus \(\tau_a(i, i)\) is an even integer between \(0\) and \(2m-2\).
  Hence
  \[
  |\{\tau(i, i): i \in \{1,\dots,n\}\}| \le 1^{|A^-|}m^{|A^+|} = m^{p/2-1}=n^{(p/2-1)/(p-1)}=n^{1/2-1/(2p-2)}.
  \]
  Now let \(i, j \in \{1,\dots,n\}\) be such that \(j-i=1\). We first show that $d(j,a)-d(i,a)\in \{-(m-1),1,0\}$ for all $a \in \{0,\dots,p-2\}$. 
  Since \(j > 1\), there exists \(b \in \{0, \ldots, p-2\}\) such that \(d(j, b) \neq 0\). 
  Let \(c_{i, j}\) denote the least element \(c \in \{0, \ldots, p-2\}\) such that \(d(j, c) \neq 0\). 
  Observe that whenever \(0 \le a < c_{i, j}\), we have \(d(i, a) = m - 1\) and \(d(j, a) = 0\).
  Furthermore, \(d(j, c_{i, j}) - d(i, c_{i, j}) = 1\), and
  \(d(j, b) = d(i, b)\) whenever \(c_{i, j} < b \le p-2\).
  In particular, the value of \(c_{i, j}\) determines the value of \(d(j, b) - d(i, b)\) for each \(b \in \{0, \ldots, p-2\}\):
  \[
  d(j, b) - d(i, b) =
  \begin{cases}
    -(m-1)&\textrm{if }b<c_{i, j},\\
    1&\textrm{if }b=c_{i, j},\\
    0&\textrm{if }b>c_{i, j}.\\
  \end{cases}
  \]
  Hence for any \(c \in \{0, \ldots, p-2\}\) and \(a \in A^-\),
  for all pairs \((i, j)\) with \(j-i=1\) and \(c_{i, j} = c\),
  the value of \(\tau_a(i, j)\) is determined and is independent of \(i\) and \(j\):
  if \(\had{k}{a} = +1\) and \(\had{\ell}{a} = -1\), then \(\tau_a(i, j) = m-1-(d(j, a) - d(i, a))\), and if \(\had{k}{a} = -1\) and \(\had{\ell}{a} = +1\), then \(\tau_a(i, j) = m-1+(d(j, a) - d(i, a))\). 

  Let \(a \in A^+\).
  If \(a < c_{i,j}\), then \(d(i, a) = m - 1\) and \(d(j, a) = 0\) so \(\{d_{k}(i, a), d_{\ell}(j, a)\} = \{0, m-1\}\) and thus \(\tau_a(i, j) = m-1\).
  If \(a = c_{i, j}\), then \(|d_{k}(i, a) - d_{\ell}(j, a)| = 1\) since \(d(j, c_{i, j}) - d(i, c_{i, j}) = 1\), and thus \(\tau_a(i, j)\) is an odd integer between \(1\) and \(2m-3\).
  If \(a > c_{i, j}\), then \(d_{k}(i, a) = d_{\ell}(j, a)\) since \(d(i, a) = d(j, a)\), and thus \(\tau_a(i, j)\) is an even integer between \(0\) and \(2m-2\).

  Summarizing, for any \(a, c \in \{0, \ldots, p-2\}\),
  \[
    |\{\tau_a(i, j): i, j \in \{1,\dots,n\}, j - i = 1, c(i, j) = c\}| \le
    \begin{cases}
    1&\textrm{if \(a \in A^-\),}\\
    m&\textrm{if \(a \in A^+\).}\\
    \end{cases}
  \]

  Hence 
  \begin{align*}
    |\{\tau(i, j): i, j \in \{1,\dots,n\}, j-i=1\}|
    &\le \sum_{c=0}^{p-2}|\{\tau(i, j): i, j \in \{1,\dots,n\}, j-i=1, c_{i, j} = c\}|\\
    &\le(p-1)1^{|A^-|}m^{|A^+|}\\
    &=(p-1)m^{p/2-1}\\
    &=(p-1)n^{1/2-1/(2p-2)}.
  \end{align*}
By a symmetric argument,
  \[
    |\{\tau(i, j): i, j \in \{1,\dots,n\}, j-i=-1\}| \le (p-1)n^{1/2-1/(2p-2)}.
  \]
  Therefore \[|\{\tau(i, j): i, j \in \{1,\dots,n\}, |j-i|\le1\}| \le (1+2(p-1))n^{1/2-1/(2p-2)} = (2p-1)n^{1/2-1/(2p-2)},\] 
  which completes the proof of the first part of the lemma. 
  
  It remains to show that for each \(k \in \{1,\dots,p\}\), the set \(E(P_n)\) can be partitioned into \(2p-3\) stacks with respect to \((\pi_k^{-1}(1), \ldots, \pi_k^{-1}(m^{p-1}))\).
  Observe that for each \(a \in \{0, \ldots, p-2\}\) and each \(j \in \{1,\dots,m^{p-1-a}\}\), the integers \((j-1)m^a+1\), \ldots, \(jm^a\) forms a block of consecutive elements in \((\pi_k^{-1}(1), \ldots, \pi_k^{-1}(n))\). Furthermore, if \(j\) is not divisible by \(m\), then the blocks corresponding to \(j\) and \(j+1\) are consecutive. Construct a partition of $E(P_n)$ into \(2p-3\) stacks as follows. The first stack consists of the edges \(i(i+1)\) such that \(i\) is not divisible by \(m\). For each \(a \in \{1,\dots,p-2\}\), partition the edges \(i(i+1)\) such that \(i\) is divisible by \(m^a\) but not by \(m^{a+1}\) into two stacks: edge \(i(i+1)\) is in the first stack if \(i/m^a\) is odd and in the second stack if \(i/m^a\) is even.   The resulting two sets of edges are indeed stacks; in fact, no two edges cross or even nest in these sets.  Thus \(E(P_n)\) can be partitioned into \((2p-3)\) stacks.
\end{proof}

\cref{thm:SnGboxtimesPn} is a consequence of the following technical variant.

\begin{lem}\label{lem:SnGboxtimesPn}
  Assume there exists a Hadamard matrix of order \(p\). Let \(m \in \mathbb{N}\) and \(n=m^{p-1}\). Let \(G\) be a \(p\)-colourable graph with a dispersable \(d\)-stack layout. Then
  \[
    \sn(G \boxtimes P_n) \le d(2p-1)n^{1/2-1/(2p-2)}+2p-3.
  \]
\end{lem}
Before proving~\cref{lem:SnGboxtimesPn}, we show that it implies \cref{thm:SnGboxtimesPn}.

\begin{proof}[Proof of \cref{thm:SnGboxtimesPn}]
  Let \(m=\lceil n^{1/(p-1)}\rceil\).
  Since \(p\) is a power of $2$, there exists a Hadamard matrix of order \(p\).
  Since \(P_n \subseteq P_{m^{p-1}}\), \cref{lem:SnGboxtimesPn} implies
  \begin{align*}
    \sn(G \boxtimes P_n) &\le \sn(G \boxtimes P_{m^{p-1}})\\
    &\le d(2p-1)m^{(p-1)(1/2 - 1/(2p-2))}+2p-3\\
    &\le d(2p-1)m^{p/2-1}+2p-3\\
    &< d(2p-1)(n^{1/(p-1)}+1)^{p/2-1}+2p-3\\
    &\le d(2p-1)(2n^{1/(p-1)})^{p/2-1}+2p-3\\
    &=2^{p/2-1}d(2p-1)n^{1/2 - 1/(2p-2)}. \qedhere
  \end{align*}
\end{proof}

\begin{proof}[Proof of \cref{lem:SnGboxtimesPn}]
  Let \(\rho\colon V(G) \to \{1,\dots,p\}\) 
  be a proper colouring of \(G\). 
  Let $\pi_1,\dots,\pi_p$ be the permutations of $\{1,\dots,n\}$ given by \cref{PermutationSums}.
  For each \(v \in V(G)\), let \(P^v\) denote the path in \(G \boxtimes P_n\) induced by \(\{v\} \times \{1,\dots,n\}\),
  and let \(\overrightarrow{P^v} = ((v, \pi_{\rho(v)}^{-1}(1)), \ldots (v, \pi_{\rho(v)}^{-1}(n)))\).
  Let \(N = |V(G)|\), and let \(((v_1, \ldots, v_N), \psi)\) be a dispersable \(d\)-stack layout of \(G\).
  Let \(\overrightarrow{V} = (\overrightarrow{P^{v_1}}; \overrightarrow{P^{v_2}}; \ldots; \overrightarrow{P^{v_N}})\) be an ordering of \(V(G \boxtimes P_{n})\). By our choice of the permutations \(\pi_1\), \ldots, \(\pi_p\), for each \(v \in V(G)\), the set \(E(P^{v})\) can be partitioned into \(2p-3\) stacks. Since the paths \(P^v\) occupy disjoint parts of \(\overrightarrow{V}\), it follows that $\bigcup_{v \in V(G)} E(P^v)$ admits a partition into \(2p-3\) stacks with respect to \(\overrightarrow{V}\).

  We partition the set \(E(G \boxtimes P_{n}) \setminus \bigcup_{v \in V(G)} E(P^v)\) into sets \(E_{u v}\) indexed by the edges of \(G\).
  For an edge $uv \in E(G)$, let $E_{uv}=\{xy \in E(G\boxtimes P_n):x \in V(P^u), y \in V(P^v)\}$. For each $xy \in E_{uv}$, let $\gamma(xy)=\psi(uv)$. We claim that for all $k \in \{1,\dots,d\}$, $\gamma^{-1}(k)$ can be partitioned into at most $(2p-1)n^{1/2-1/(2p-2)}$ stacks with respect to $\overrightarrow{V}$. Since $\psi$ is a dispersable stack-layout, it suffices to show that for every $uv\in E(G)$, $E_{uv}$ can be partitioned into at most $(2p-1)n^{1/2-1/(2p-2)}$ stacks.

  Fix an edge \(u v \in E(G)\). For each edge \(e = (u, i)(v, j) \in E_{uv}\), define
  \[
    \phi(e) = \pi_{\rho(u)}(i) + \pi_{\rho(v)}(j).
  \]
  By our choice of the permutations \(\pi_1\), \ldots, \(\pi_p\), the size of the image of \(\phi\) is at most \((2p-1)n^{1/2-1/(2p-2)}\). It remains to show that $\phi$ partitions $E_{uv}$ into stacks with respect to \(\overrightarrow{V}\). Let \(e = (u, i)(v, j)\) and \(e' = (u, i')(v, j')\) be two edges from \(E_{uv}\) which cross. Without loss of generality, assume that the vertices are in the order \((u, i)\), \((u', i')\), \((v, j)\), \((v', j')\)
  in \(\overrightarrow{V}\).
  This means that \(\pi_{\rho(u)}(i) < \pi_{\rho(u)}(i')\) and \(\pi_{\rho(v)}(j) < \pi_{\rho(v)}(j')\), so
  \(\phi(e) = \pi_{\rho(u)}(i) + \pi_{\rho(v)}(j) < \pi_{\rho(u)}(i') + \pi_{\rho(v)}(j') = \phi(e')\), so \(\phi(e) \neq \phi(e')\). Hence \(\phi\) partitions \(E_{u v}\) into $(2p-1)n^{1/2-1/(2p-2)}$ stacks, as required.
\end{proof}

Note that for values of $p$ that are not powers of 2 but there exists a Hadamard matrix of order $p$, 
	\cref{lem:SnGboxtimesPn} gives a stronger bound than 
	\cref{thm:SnGboxtimesPn}.

\section{Smaller maximum degree}
\label{MaxDegree}

This section proves \cref{t:Maxsnqn}, which says that if $\Delta_0$ is the minimum integer for which there exists a graph family with maximum degree $\Delta_0$, unbounded stack-number and bounded queue-number, then $\Delta_0\in\{6,7\}$. These upper and lower bounds are respectively proved in \cref{DeltaLe7,DeltaGt5} below. 

\begin{thm}\label{DeltaLe7}
	There exists a graph family with maximum degree \(7\), unbounded stack-number and bounded queue-number. 
\end{thm}

The construction for \cref{DeltaLe7} is based on a tessellation of $\mathbb{R}^3$ with truncated octahedra, first studied by \citet{Fedorov}. Let \defn{\(Q_0\)} denote the convex hull of all points \((x, y, z) \in \mathbb{R}^3\) such that \(\{|x|, |y|, |z|\} = \{0, 1, 2\}\); see \cref{TruncatedOctahedron}.
A \defn{truncated octahedron} is any polyhedron similar to \(Q_0\). At each corner of a truncated octahedron three faces meet: one square and two regular hexagons. Let \defn{\(Q_0 + (x, y, z)\)} be the translation of \(Q_0\) by a vector \((x, y, z)\).

\begin{figure}[!h]
	\includegraphics{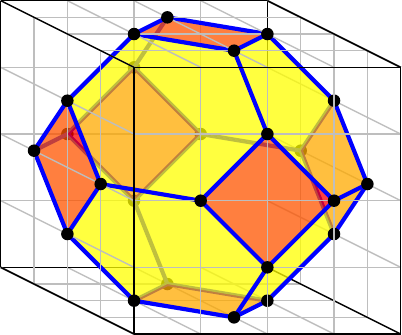}
	\caption{The truncated octahedron \(Q_0\) inscribed in a cube.}
	\label{TruncatedOctahedron}	
\end{figure}

Let \(\mathcal{T}_\infty\) be the family of translations of \(Q_0\) defined as 
\[
\mathcal{T}_\infty = \{Q_0+(4x, 4y, 4z): {x, y, z \in \mathbb{Z}}\} \cup \{Q_0+(4x+2, 4y+2, 4z+2) : {x, y, z \in \mathbb{Z}}\}.
\]
\(\mathcal{T}_\infty\) is a \defn{3-dimensional tessellation}; that is, a family of interior-disjoint polyhedra whose union is \(\mathbb{R}^3\).
A \defn{corner} of \(\mathcal{T}_\infty\) is any corner of a truncated octahedron from \(\mathcal{T}_\infty\); \defn{edges} and \defn{faces} of \(\mathcal{T}_\infty\) are defined similarly. 
Every (hexagonal or square) face of \(\mathcal{T}_\infty\) is shared by two truncated octahedra in \(\mathcal{T}_\infty\).
At each edge of \(\mathcal{T}_\infty\) a square face and two hexagonal faces meet, and at each corner of \(\mathcal{T}_\infty\) two square and four hexagonal faces meet.

We construct an infinite graph \(G_\infty\) whose vertices are points in \(\mathbb{R}^3\) and edges are line segments between their endpoints. \(G_\infty\) is the union of copies of the plane graphs \(\Jfour\) and \(\Jsix\) (depicted in \cref{Triangulations}), where the copies of \(\Jfour\) are placed at the square faces of \(\mathcal{T}_\infty\) and the copies of \(\Jsix\) are placed at the hexagonal faces of \(\mathcal{T}_\infty\). Each copy of \(\Jfour\) or \(\Jsix\) is contained within its corresponding face so that the exterior cycle coincides with the union of the edges contained in that face, black vertices are at the corners of the face, and each edge of the face is split into equal segments by \(10\) vertices from the exterior face.
A vertex of \(G_\infty\) is called a \defn{corner vertex} if it coincides with a corner of \(\mathcal{T}_\infty\), an \defn{edge vertex} if it lies on an edge of \(\mathcal{T}_\infty\) and is not a corner vertex, or a \defn{face vertex} if it is neither a corner vertex nor an edge vertex.
\begin{figure}
	\includegraphics{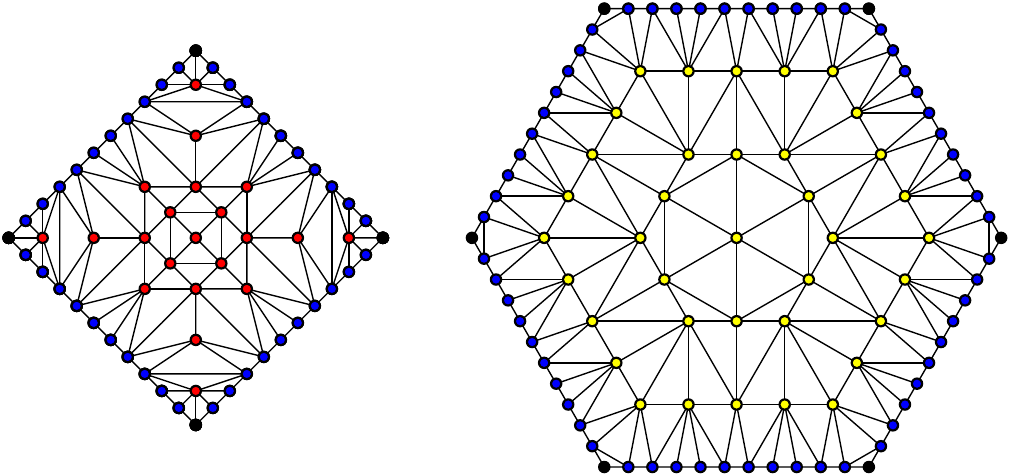}
	\caption{The graphs \(\Jfour\) and \(\Jsix\).}
	\label{Triangulations}
\end{figure}

\begin{lem}\label{l:DeltaGinftyEq7}
	The maximum degree of \(G_\infty\) is \(7\).
\end{lem}

\begin{proof}
	Let \(v\in V(G_\infty)\). We proceed by case analysis. If \(v\) is a face vertex, then its degree is at most \(7\) because \(\Jfour\) and \(\Jsix\) have maximum degree \(7\).
	
	Now suppose \(v\) is a corner vertex. Since \(v\) belongs to four edges of \(\mathcal{T}_\infty\), it is adjacent to four edge vertices in \(G_{\infty}\). Furthermore, $v$ is adjacent to a face vertex in each of the two copies of \(\Jfour\) containing \(v\) and is not adjacent to any face vertex in a copy of \(\Jsix\). Therefore, \(v\) has degree \(6\). 
	
	It remains to consider the case when $v$ is an edge vertex. Let \(\varepsilon\) be the edge of the tessellation \(\mathcal{T}_\infty\) that contains \(v\). 
	Let \(v_0 \cdots v_{11}\) be the path induced by the vertices of \(G_\infty\) which lie on \(\varepsilon\),
	so that \(v_0\) and \(v_{11}\) are corner vertices and \(v_1\), \ldots, \(v_{10}\) are edge vertices (and one of them is \(v\)).
	The vertices \(v_1\), \ldots, \(v_{10}\) and all their neighbours lie in one copy of \(\Jfour\) and two copies of \(\Jsix\).
	The degrees of the vertices \(v_1\), \ldots, \(v_{10}\) in the copy of \(\Jfour\) containing them are
	\(3\), \(3\), \(5\), \(5\), \(3\), \(3\), \(5\), \(5\), \(3\), \(3\), respectively,
	and their degrees in each copy of \(\Jsix\) containing them are
	\(4\), \(4\), \(3\), \(3\), \(4\), \(4\), \(3\), \(3\), \(4\), \(4\), respectively.
	Since \(3 + 2\cdot4=5+2\cdot3=11\),
	for each \(i \in \{1, \ldots, 10\}\), 
	the total sum of degrees of \(v_i\) in the copy of \(\Jfour\) and the two copies of \(\Jsix\) is \(11\). However, we counted the vertices \(v_{i-1}\) and \(v_{i+1}\) thrice, so the degree of \(v_i\) in \(G_\infty\) is \(11 - 4 = 7\), as required.
\end{proof}

Let \defn{\([a, b]\)} denote the closed interval \(\{x \in \mathbb{R} : a \le x \le b\}\).	Observe that for every \((i, j, k) \in \mathbb{Z}^3\), the cube
 \( [4i-2,4i+2]\times [4j-2,4j+2] \times [4k-2,4k+2]\) contains exactly \(|V(G_\infty) \cap Q_0|\) vertices of \(V(G_\infty)\). For every \(n\in \mathbb{N}\), let \(\mathcal{F}_n\) be the set of all faces of \(\mathcal{T}_\infty\) contained in \([4, 4n+2]^3\) and let \(G_n\) be the subgraph of \(G_\infty\) induced by the vertices lying on the faces in \(\mathcal{F}_n\). Then \(|V(G_n)| \in \Theta(n^3)\). Furthermore, \((G_n)_{n \in \mathbb{N}}\) has cubic growth as the distance between any pair of adjacent vertices in \(G_n\) is \(O(1)\).
By \cref{c:GrowthQueue} and \cref{l:DeltaGinftyEq7}, $(G_n)_{n\in \mathbb{N}}$ has bounded queue-number and maximum degree \(7\). Thus \cref{DeltaLe7} follows from the next lemma.

\begin{lem}\label{SnGn}
	\(\sn(G_n) \in \Omega(n^{1/3})\).
\end{lem}
\begin{proof}
\(\Jfour\) and \(\Jsix\) are plane graphs in which every internal face is a triangle. Therefore, every subgraph induced by a face of \(\mathcal{T}_\infty\) has a triangle complex homeomorphic to that face (and to a closed disk). Furthermore, since faces intersect only at their edges and corners, every subgraph of \(G_n\) induced by a union of faces from \(\mc{F}_n\) has a triangle complex homeomorphic to that union.
	
	For \(a \in \{1, 2, 3\}\) and \(i \in \{1,\dots,n\}\), let \(X^a_i\) be the union of all faces \(F \in \mathcal{F}_n\) such that \(4i \le x_a \le 4i + 2\) for all \((x_1, x_2, x_3) \in F\). Each set \(X^a_i\) is homeomorphic to a closed disk (see \cref{tessellationLayer}). Let \(L^a_i\) denote the subgraph of \(G_n\) induced by \(X^a_i\).
	As observed earlier, \(X^a_i\) as a union of faces from \(\mc{F}_n\) is homeomorphic to \(\Tr(L^a_i)\).
	Hence we identify \(X^a_i\) with \(\Tr(L^a_i)\). 
	\begin{figure}[h!]	
		\includegraphics{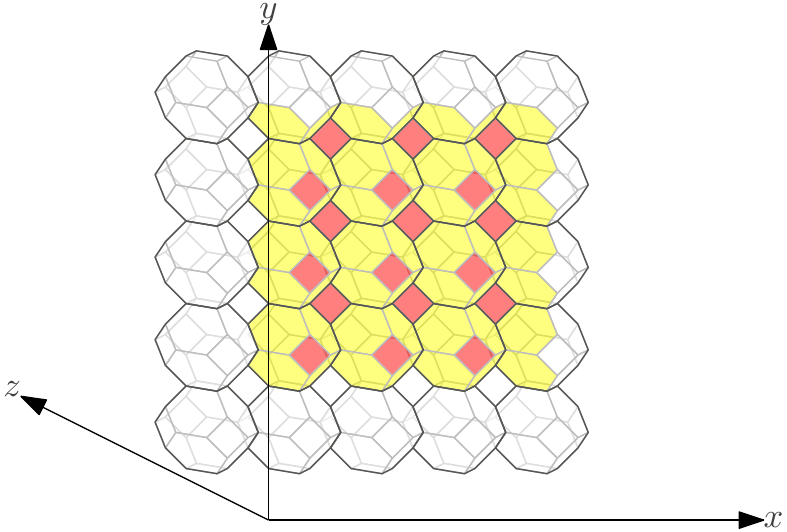}
		\caption{The set \(X^a_i\) for \(n = 4\), \(a=3\) and \(i=2\). The truncated octahedra are of the form \(Q_0 + (4i+2, 4j+2, 4k-2)\) for \((j, k) \in \{0, \ldots, n\}^2\).}
		\label{tessellationLayer}	
	\end{figure}
	
	Observe that there are \((2n-1)^3\) hexagonal faces in \(\mathcal{F}_n\), each of which is contained in a different cube of the form \([2x+2, 2x+4] \times [2y+2, 2y+4] \times [2z+2, 2z+4]\) with \((x, y, z) \in \{1,\dots,2n-1\}^3\).
	Furthermore, each square face has its centre in a corner of one of these cubes (but not every corner is a centre of a square face). Hence, for \((i, j) \in \{1,\dots,n\}^2\), the intersection \(X^1_i \cap X^2_j\) is the union of the \(2n-1\) hexagonal faces from \(\mathcal{F}_n\) contained in \([4i, 4i+2] \times [4j, 4j+2] \times \mathbb{R}\). These hexagons form a sequence such that any pair of consecutive hexagons share an edge. The intersection \(X^1_i \cap X^2_j\) is thus connected. Symmetric arguments show that
	\begin{equation}\label{e:isection2}
	\textrm{\(X^a_i \cap X^b_j\) is connected for \((i, j) \in \{1,\dots,n\}^2\) and \((a, b) \in \{1, 2, 3\}^2\) such that \(a \neq b\)}.
	\end{equation}
	Furthermore, for \((i, j, k) \in \{1,\dots,n\}^3\), the intersection \(X^1_i \cap X^2_j \cap X^3_k\) is the hexagonal face contained in \([4i, 4i+2] \times [4j, 4j+2] \times [4k, 4k+2]\), so in particular
	\begin{equation}\label{e:isection3}
	\textrm{\(X^1_i \cap X^2_j \cap X^3_k\) is connected (and non-empty) for \((i, j, k) \in \{1,\dots,n\}^3\)}.
	\end{equation}
	
	Let \(\mc{B} = \{B(i, j, k)\}_{(i, j, k) \in \{1,\dots,n\}^3}\) where \(B(i, j, k) = X^1_i \cup X^2_j \cup X^3_k\).
	We claim that \(\mc{B}\) is a bramble over \(\Tr(G_n)\). We proceed by verifying the preconditions of \cref{c:bramble}.
	
	Let \((i, j, k) \in \{1,\dots,n\}^3\). We first show that \(B(i, j, k)\) is simply connected. We have \(B(i, j, k) = X^1_i \cup X^2_j \cup X^3_k\), and each of \(X^1_i\), \(X^2_j\) and \(X^3_k\) is homeomorphic to a closed disk and is thus simply connected.
	By \cref{e:isection2} and \cref{e:isection3}, the sets \(X^1_i \cap X^2_j\) and \(X^1_i \cap X^2_j \cap X^3_k\) are connected. Hence, by \cref{l:vanKampen}, \(B(i, j, k)\) is simply connected.
	
	Now, let \((i_1, i_2, i_3), (j_1, j_2, j_3) \in \{1,\dots,n\}^3\), and let \(B_1 = B(i_1, i_2, i_3)\) and \(B_2 = B(j_1, j_2, j_3)\). 
	We now show that \(B_1 \cap B_2\) is connected. We have
	\begin{align*}
	B_1 \cap B_2 = (X^1_{i_1} \cup X^2_{i_2} \cup X^3_{i_3}) \cap (X^1_{j_1} \cup X^2_{j_2} \cup X^3_{j_3})
	= \bigcup_{a=1}^3 \bigcup_{b = 1}^3 \left(X^{a}_{i_{a}} \cap X^{b}_{j_b}\right).
	\end{align*}
	
	By \cref{e:isection2}, \(X^a_{i_a} \cap X^b_{j_b}\) is connected when \(a \neq b\).
	Furthermore, if \(\{a, b, c\} = \{1, 2, 3\}\), then, by~\cref{e:isection3}, \(X^a_{i_a} \cap X^b_{j_b}\) intersects \(X^{a}_{i_{a}} \cap X^{c}_{j_{c}}\) and \(X^c_{i_c} \cap X^b_{j_b}\). Hence, the union of all \(X^a_{i_a} \cap X^{b}_{j_b}\) with \(a \neq b\) is connected.
	Furthermore, for each \(a \in \{1, 2, 3\}\), the intersection \(X^a_{i_a} \cap X^a_{j_a}\) is \(X^a_{i_a}\) if \(i_a = j_a\), or an empty set if \(i_a \neq j_a\).
	In the former case, \(X^a_{i_a} \cap X^a_{j_a}\) intersects \(X^a_{i_a} \cap X^b_{j_b}\) for \(b \in \{1, 2, 3\} \setminus \{a\}\).
	Hence, \(B_1 \cap B_2\) is indeed connected.
	
	Finally, for any \((i_1, j_1, k_1), (i_2, j_2, k_2), (i_3, j_3, k_3) \in \{1,\dots,n\}^3\), the intersection \(B(i_1, j_1, k_1) \cap B(i_2, j_2, k_2) \cap B(i_3, j_3, k_3)\) contains \(X^1_{i_1} \cap X^2_{j_2} \cap X^3_{k_3}\) which is non-empty by \cref{e:isection3}. Hence, by \cref{c:bramble}, \(\mathcal{B}\) is a bramble.
	
	By definition, \(|\mathcal{B}| = n^3\). Moreover, for every \(p \in \Tr(G_n)\) and every \(a \in \{1, 2, 3\}\) there exists at most one \(i \in \{1,\dots,n\}\) such that \(p \in \Tr(L^a_i)\). Thus \(\congestion(\mc{B}) \le 3n^2\). By \cref{l:bramble}, 
	\[
	\overlap(\Tr(G_n), \mathbb{R}^2) \ge \beta \frac{|\mc{B}|}{\congestion(\mathcal{B})} = \beta \frac{n^3}{3n^2} = \frac{\beta}{3}n.
	\]
	Since the maximum degree of \(G_n\) is \(7\), each vertex is contained in at most \(\binom{7}{2} = 21\) triangles. Therefore, by \cref{l:SNvsOverlap}, 
	\[
	\sn(G_n) \ge \left(\frac{\overlap(\Tr(G_n), \mathbb{R}^2)}{63}\right)^{1/3} \ge \left(\frac{\beta}{189}n\right)^{1/3}.\qedhere
	\]
\end{proof}

We now prove the lower bound for $\Delta_0$ (defined at the start of the section).

\begin{thm}\label{DeltaGt5}
	Every graph class with maximum degree \(5\) and bounded queue-number has bounded stack-number. 
\end{thm}

The proof of \cref{DeltaGt5} depends on the following definitions. For $k,c\in\mathbb{N}$, a graph $G$ is \defn{$k$-colourable with clustering $c$} if each vertex of $G$ can be assigned one of $k$ colours such that each monochromatic component has at most $c$ vertices. Here a \defn{monochromatic component} is a maximal monochromatic connected subgraph. The \defn{clustered chromatic number} of a graph class $\mathcal{G}$ is the minimum $k\in\mathbb{N}$ such that for some $c\in\mathbb{N}$ every graph in $\mathcal{G}$ is $k$-colourable with clustering $c$. See \citep{Wood18} for a survey on clustered graph colouring. \citet{HST03} proved that the class of graphs with maximum degree at most 5 has clustered chromatic number 2 (which is best possible, since $(P_n\boxslash P_n)_{n \in \mathbb{N}}$ has maximum degree \(6\) and clustered chromatic number 3 by the Hex Lemma). Thus, \cref{DeltaGt5} is an immediate consequence of the following result. 

\begin{thm}
\label{ClusteredQnSn}
Every graph class $\mathcal{G}$ with bounded queue-number and clustered chromatic number at most 2 has bounded stack-number. 
\end{thm}

The proof of \cref{ClusteredQnSn} depends on the following lemmas.

\begin{lem}[\citep{Pemmaraju92,DPW04}]
	\label{BipartiteSnQn}
	For every bipartite graph $G$, $$\sn(G)\leq 2\qn(G).$$
\end{lem}

\citet{GH01} proved that every graph of treewidth $k$ has stack-number at most $k+1$. A close inspection of their proof actually shows the following result. A \defn{star forest} is a forest where each component is isomorphic to a star. In a \(k\)-vertex component of a star forest with \(k \neq 2\), there is a unique vertex of degree \(k-1\) called its \defn{centre}. In a \(2\)-vertex component, any vertex may be chosen as the \defn{centre}.

\begin{lem}[\cite{GH01}]\label{l:twstar}
	Every graph $G$ of treewidth $k$ has an edge-partition into
	$k+1$ spanning star forests $G_1,\dots,G_{k+1}$, such that each vertex
	is the centre of a star component of some $G_i$.
\end{lem}

\begin{lem}
	\label{StackCompleteProduct}
	For every graph $G$ and $t\in\mathbb{N}$, 
	$$\sn(G\boxtimes K_t) \leq \max\{ 3t\,\sn(G), \lceil\tfrac{t}{2}\rceil\}.$$
\end{lem}

\begin{proof}	
	We may assume that $G$ is connected and $V(K_t) = \{1,\dots,t\}$. If $|V(G)|=1$ then the claim holds since $\sn(G)=0$ and $\sn(G\boxtimes K_t)=\sn(K_t)\leq \lceil \tfrac{t}{2} \rceil$; see \cite{BK79}. 
	Now assume that $|V(G)|\geq 2$ and thus $E(G)\neq \emptyset$. Let $s=\sn(G)\geq 1$. Let $(v_1,\dots,v_n)$ together with $\psi \colon E(G)\to \{1,\dots,s\}$ be an $s$-stack layout of $G$. For each $k \in \{1,\dots,s\}$, let $G_k$ be the spanning subgraph of $G$ with $E(G_k)=\psi^{-1}(k)$. Thus $G_k$ admits a $1$-stack layout and is therefore outerplanar. Since outerplanar graphs have treewidth at most $2$, it follows by \cref{l:twstar} that $G_k$ has an edge-partition into three spanning star forests $G_{k,1},G_{k,2},G_{k,3}$, such that for each vertex $v\in V(G)$, there is some $a\in \{1,2,3\}$ such that $v$ is the centre of a star component of $G_{k,a}$. 
		
	Define $\phi \colon E(G\boxtimes K_t) \to \{1,\dots,s\}\times\{1,2,3\}\times\{1,\dots,t\}$ as follows. 	Consider an edge $e=(u,i)(v,j)$ of $G\boxtimes K_t$. If $u=v$ then let $(k,a)\in \{1,\dots,s\}\times\{1,2,3\}$ be such that $u$ is the centre of a star component of $G_{k,a}$. Let $\phi(e)=(k,a,i)$. Otherwise $uv\in E(G_{k,a})$ for some $k\in\{1,\dots,s\}$ and $a\in\{1,2,3\}$. Since $G_{k,a}$ is a star-forest, without loss of generality, $u$ is the centre of the star component of $G_{k,a}$ containing $uv$. Let $\phi(e)=(k,a,i)$. 
	
	We claim that $\phi$ and the vertex-ordering 
	\[((v_1,1),\dots,(v_1,t);(v_2,1),\dots,(v_2,t);\dots;(v_n,1),\dots,(v_n,t))\] 
	define a $3st$-stack layout of $G \boxtimes K_{t}$ where $\phi^{-1}((k,a,i))$ is a non-crossing star-forest for each $(k,a,i)\in \{1,\dots,s\}\times \{1,2,3\}\times \{1,\dots,t\}$. 
	Consider edges $e$ and $e'$ of $G \boxtimes K_t$ with $\phi(e)=\phi(e')=(k,a,i)$ for some $(k,a,i)\in\{1,\dots,s\}\times \{1,2,3\}\times \{1,\dots,t\}$. 
	By construction, $e=(u,i)(v,j)$ and $e'=(x,i)(y,\ell)$ for some $u,v,x,y\in V(G)$ and $j,\ell\in\{1,\dots,t\}$. 
	If $u\neq v$ then $uv$ is an edge of some component of $G_{k,a}$ centred at $u$. 
	Similarly, if $x\neq y$ then $xy$ is an edge of some component of $G_{k,a}$ centred at $x$. 
	If $u$ and $x$ are in the same component of $G_{k,a}$, then $e$ and $e'$ have a common end-vertex $(u,i)=(x,i)$ and therefore do not cross. 
	Now assume $u$ and $x$ are in distinct components of $G_{k,a}$. 
	If $u \neq v$ and $x\neq y$, then $\psi(uv)=\psi(xy)=k$ and $uv$ and $xy$ do not cross, implying $e$ and $e'$ do not cross by the choice of vertex-ordering. 
	Now assume that $u=v$ or $x=y$. Since $\{u,v\}\cap \{x,y\}= \emptyset$, $e$ and $e'$ do not cross by the choice of vertex-ordering.
\end{proof}

\begin{proof}[Proof of \cref{ClusteredQnSn}]
	By assumption, there exist $c,\ell\in\mathbb{N}$ such that for every graph $G\in\mathcal{G}$, we have $\qn(G)\leq c$ and $G$ is 2-colourable with each monochromatic component having at most $\ell$ vertices. Contracting each monochromatic component to a single vertex gives a bipartite $\ell$-small minor $H$ of $G$. \cref{ShallowMinorQueue} implies $\qn(H) \leq 2\ell(2c)^{2\ell}$. \cref{BipartiteSnQn} implies $\sn(H) \leq 4\ell(2c)^{2\ell}$. By construction, $G$ is isomorphic to a subgraph of $H \boxtimes K_{\ell}$. Thus $\sn(G) \leq \sn(H\boxtimes K_\ell)$, which is at most $12\ell^2(2c)^{2\ell}$ by \cref{StackCompleteProduct}. Hence $\mathcal{G}$ has bounded stack-number. 
\end{proof}

\section{Open Problems}

We finish with some open problems:
\begin{itemize}
\item Does there exist a graph class with bounded stack-number and unbounded queue-number? This is equivalent to the question of whether graphs with stack-number 3 have bounded queue-number~\citep{DujWoo05}.
\item Do graphs with queue-number 2 (or 3) have bounded stack-number~\citep{DEHMW21}?
\item Do graph classes with quadratic (or linear) growth have bounded stack-number?
	\item Does there exist a graph family with unbounded stack-number, bounded queue-number and maximum degree $6$?
	\item The best known lower bound on the maximum stack-number of $n$-vertex graphs with fixed maximum degree $\Delta$ is $\Omega(n^{1/2-1/\Delta})$, proved by \citet{Malitz94a} using a probabilistic argument. Is there a constructive proof of this bound?
	\item The best upper bound on the stack-number of $n$-vertex graphs with fixed maximum degree $\Delta$ is $O(n^{1/2})$, also due to \citet{Malitz94a}. Closing the gap between the lower and upper bounds is an interesting open problem. For example, the best bounds for graphs of maximum degree 3 are $\Omega(n^{1/6})$ and $O(n^{1/2})$. 	
\end{itemize}

\subsection*{Acknowledgements.} This work was initiated at the \emph{Workshop on Graph Product Structure Theory} (BIRS21w5235) at the Banff International Research Station, 21--26 November 2021. Thanks to the other organisers and participants. 

{
\fontsize{10pt}{10pt}

\begin{thebibliography}{62}
	\providecommand{\natexlab}[1]{#1}
	\providecommand{\msn}[1]{MR:\,\href{http://www.ams.org/mathscinet-getitem?mr=MR{#1}}{#1}}
	\providecommand{\ZBL}[1]{Zbl:\,\href{https://www.zentralblatt-math.org/zmath/en/search/?q=an:#1}{#1}}
	\providecommand{\url}[1]{\texttt{#1}}
	\providecommand{\urlprefix}{}
	\expandafter\ifx\csname urlstyle\endcsname\relax
	\providecommand{\doi}[1]{doi:\discretionary{}{}{}#1}\else
	\providecommand{\doi}{doi:\discretionary{}{}{}\begingroup
		\urlstyle{rm}\Url}\fi
	
	\bibitem[{Alam et~al.(2015)Alam, Brandenburg, and Kobourov}]{ABK15}
	\textsc{Muhammad~J. Alam, Franz~J. Brandenburg, and Stephen~G. Kobourov}.
	\newblock \href{http://arxiv.org/abs/1510.05891}{On the book thickness of
		1-planar graphs}.
	\newblock 2015, arXiv:1510.05891.
	
	\bibitem[{Alam et~al.(2021)Alam, Bekos, Dujmovi\'c, Gronemann, Kaufmann, and
		Pupyrev}]{ABDGKP21}
	\textsc{Muhammad~Jawaherul Alam, Michael~A. Bekos, Vida Dujmovi\'c, Martin
		Gronemann, Michael Kaufmann, and Sergey Pupyrev}.
	\newblock \href{https://doi.org/10.1016/j.tcs.2021.01.035}{On dispersable book
		embeddings}.
	\newblock \emph{Theor. Comput. Sci.}, 861:1--22, 2021.
	
	\bibitem[{Ambrus et~al.(2006)Ambrus, Bar{\'a}t, and Hajnal}]{ABH06}
	\textsc{Gergely Ambrus, J{\'a}nos Bar{\'a}t, and P{\'e}ter Hajnal}.
	\newblock \href{http://www.acta.hu/}{The slope parameter of graphs}.
	\newblock \emph{Acta Sci. Math. (Szeged)}, 72(3--4):875--889, 2006.
	
	\bibitem[{Bar{\'a}t et~al.(2006)Bar{\'a}t, Matou{\v{s}}ek, and Wood}]{BMW06}
	\textsc{J{\'a}nos Bar{\'a}t, Ji{\v{r}}{\'i} Matou{\v{s}}ek, and David~R. Wood}.
	\newblock \href{https://doi.org/10.37236/1029}{Bounded-degree graphs have
		arbitrarily large geometric thickness}.
	\newblock \emph{Electron. J. Combin.}, 13(1):R3, 2006.
	
	\bibitem[{Baur and Brandes(2004)}]{BB04}
	\textsc{Michael Baur and Ulrik Brandes}.
	\newblock \href{https://doi.org/10.1007/978-3-540-30559-0\_28}{Crossing
		reduction in circular layouts}.
	\newblock In \emph{Proc. 30th Int'l Workshop on Graph-Theoretic Concepts in
		Comput. Sci. (WG '04)}, vol. 3353 of \emph{Lecture Notes in Comput. Sci.},
	pp. 332--343. Springer, 2004.
	
	\bibitem[{Beame et~al.(2009)Beame, Blais, and Huynh-Ngoc}]{BBHN}
	\textsc{Paul Beame, Eric Blais, and Dang-Trinh Huynh-Ngoc}.
	\newblock \href{http://arxiv.org/abs/0904.1615}{Longest common subsequences in
		sets of permutations}.
	\newblock 2009, arXiv:0904.1615.
	
	\bibitem[{Bekos et~al.(2017)Bekos, Bruckdorfer, Kaufmann, and
		Raftopoulou}]{BBKR17}
	\textsc{Michael~A. Bekos, Till Bruckdorfer, Michael Kaufmann, and Chrysanthi~N.
		Raftopoulou}.
	\newblock \href{https://doi.org/10.1007/s00453-016-0203-2}{The book thickness
		of 1-planar graphs is constant}.
	\newblock \emph{Algorithmica}, 79(2):444--465, 2017.
	
	\bibitem[{Bekos et~al.(2020{\natexlab{a}})Bekos, Da~Lozzo, Griesbach,
		Gronemann, Montecchiani, and Raftopoulou}]{BLGGMR20}
	\textsc{Michael~A. Bekos, Giordano Da~Lozzo, Svenja~M. Griesbach, Martin
		Gronemann, Fabrizio Montecchiani, and Chrysanthi Raftopoulou}.
	\newblock \href{https://doi.org/10.4230/LIPIcs.SoCG.2020.16}{Book embeddings of
		nonplanar graphs with small faces in few pages}.
	\newblock In \textsc{Sergio Cabello and Danny~Z. Chen}, eds., \emph{Proc. 36th
		Int'l Symp. Computat. Geom. (SoCG '20)}, vol. 164 of \emph{LIPIcs. Leibniz
		Int. Proc. Inform.}, pp. 16:1--16:7. Schloss Dagstuhl. Leibniz-Zent. Inform.,
	2020{\natexlab{a}}.
	
	\bibitem[{Bekos et~al.(2020{\natexlab{b}})Bekos, Kaufmann, Klute, Pupyrev,
		Raftopoulou, and Ueckerdt}]{BKKPRU20}
	\textsc{Michael~A. Bekos, Michael Kaufmann, Fabian Klute, Sergey Pupyrev,
		Chrysanthi Raftopoulou, and Torsten Ueckerdt}.
	\newblock \href{https://doi.org/10.20382/jocg.v11i1a12}{Four pages are indeed
		necessary for planar graphs}.
	\newblock \emph{J. Comput. Geom.}, 11(1):332--353, 2020{\natexlab{b}}.
	
	\bibitem[{Bernhart and Kainen(1979)}]{BK79}
	\textsc{Frank~R. Bernhart and Paul~C. Kainen}.
	\newblock \href{https://doi.org/10.1016/0095-8956(79)90021-2}{The book
		thickness of a graph}.
	\newblock \emph{J. Combin. Theory Ser. B}, 27(3):320--331, 1979.
	
	\bibitem[{Bj{\"o}rner(1995)}]{Bjorner95}
	\textsc{Anders Bj{\"o}rner}.
	\newblock Topological methods.
	\newblock In \textsc{M.~Gr{\"o}tschel R.~L.~Graham and L.~Lov{\'a}sz}, eds.,
	\emph{Handbook of combinatorics}, vol.~2, chap.~34, pp. 1819--1872. North
	Holland, Amsterdam, 1995.
	
	\bibitem[{Blankenship(2003)}]{Blankenship-PhD03}
	\textsc{Robin Blankenship}.
	\newblock
	\href{http://etd.lsu.edu/docs/available/etd-0709103-163907/unrestricted/Blankenship_dis.pdf}{Book
		embeddings of graphs}.
	\newblock Ph.D. thesis, Department of Mathematics, Louisiana State University,
	U.S.A., 2003.
	
	\bibitem[{Blankenship and Oporowski(1999)}]{BO99}
	\textsc{Robin Blankenship and Bogdan Oporowski}.
	\newblock
	\href{https://citeseerx.ist.psu.edu/viewdoc/summary?doi=10.1.1.36.4358}{Drawing
		subdivisions of complete and complete bipartite graphs on books}.
	\newblock Tech. Rep. 1999-4, Department of Mathematics, Louisiana State
	University, U.S.A., 1999.
	
	\bibitem[{Bourgain(2009)}]{Bourgain09}
	\textsc{Jean Bourgain}.
	\newblock \href{https://doi.org/10.1016/j.crma.2009.02.009}{Expanders and
		dimensional expansion}.
	\newblock \emph{C. R. Math. Acad. Sci. Paris}, 347(7-8):357--362, 2009.
	
	\bibitem[{Bourgain and Yehudayoff(2013)}]{BY13}
	\textsc{Jean Bourgain and Amir Yehudayoff}.
	\newblock \href{https://doi.org/10.1007/s00039-012-0200-9}{Expansion in
		{SL$_2(\mathbb{R})$} and monotone expansion}.
	\newblock \emph{Geom. Funct. Anal.}, 23(1):1--41, 2013.
	
	\bibitem[{Brandenburg(2020)}]{B20}
	\textsc{Franz~J. Brandenburg}.
	\newblock \href{http://arxiv.org/abs/2012.06874}{Book embeddings of $k$-map
		graphs}.
	\newblock 2020, arXiv:2012.06874.
	
	\bibitem[{Buss and Shor(1984)}]{BS84}
	\textsc{Jonathan~F. Buss and Peter Shor}.
	\newblock \href{https://doi.org/10.1145/800057.808670}{On the pagenumber of
		planar graphs}.
	\newblock In \emph{Proc. 16th ACM Symp. Theory Comput. (STOC '84)}, pp.
	98--100. ACM, 1984.
	
	\bibitem[{Camarena(2021)}]{Camarena}
	\textsc{Omar~Antol{\'i}n Camarena}.
	\newblock \href{https://www.matem.unam.mx/~omar/groupoids/vankampen.pdf}{The
		{V}an {K}ampen theorem}.
	\newblock 2021.
	
	\bibitem[{Chung et~al.(1987)Chung, Leighton, and Rosenberg}]{CLR87}
	\textsc{Fan R.~K. Chung, F.~Thomson Leighton, and Arnold~L. Rosenberg}.
	\newblock \href{https://doi.org/10.1137/0608002}{Embedding graphs in books: a
		layout problem with applications to {V}{L}{S}{I} design}.
	\newblock \emph{SIAM J. Algebraic Discrete Methods}, 8(1):33--58, 1987.
	
	\bibitem[{{\Dbar}okovi\'{c} et~al.(2014){\Dbar}okovi\'{c}, Golubitsky, and
		Kotsireas}]{Hadamard}
	\textsc{Dragomir~\v{Z}. {\Dbar}okovi\'{c}, Oleg Golubitsky, and Ilias~S.
		Kotsireas}.
	\newblock \href{https://doi.org/10.1002/jcd.21358}{Some new orders of
		{H}adamard and skew-{H}adamard matrices}.
	\newblock \emph{J. Combin. Des.}, 22(6):270--277, 2014.
	
	
	\bibitem[{Dillencourt et~al.(2000)Dillencourt, Eppstein, and
		Hirschberg}]{DEH00}
	\textsc{Michael~B. Dillencourt, David Eppstein, and Daniel~S. Hirschberg}.
	\newblock \href{https://doi.org/10.7155/jgaa.00023}{Geometric thickness of
		complete graphs}.
	\newblock \emph{J. Graph Algorithms Appl.}, 4(3):5--17, 2000.
	
	\bibitem[{Dujmovi\'c et~al.(2021)Dujmovi\'c, Eppstein, Hickingbotham, Morin,
		and Wood}]{DEHMW21}
	\textsc{Vida Dujmovi\'c, David Eppstein, Robert Hickingbotham, Pat Morin, and
		David~R. Wood}.
	\newblock \href{https://doi.org/10.1007/s00493-021-4585-7}{Stack-number is not
		bounded by queue-number}.
	\newblock \emph{Combinatorica}, 2021.
	
	\bibitem[{Dujmovi{\'c} et~al.(2007{\natexlab{a}})Dujmovi{\'c}, Eppstein,
		Suderman, and Wood}]{DESW07}
	\textsc{Vida Dujmovi{\'c}, David Eppstein, Matthew Suderman, and David~R.
		Wood}.
	\newblock \href{https://doi.org/10.1016/j.comgeo.2006.09.002}{Drawings of
		planar graphs with few slopes and segments}.
	\newblock \emph{Comput. Geom. Theory Appl.}, 38:194--212, 2007{\natexlab{a}}.
	
	\bibitem[{Dujmovi\'{c} et~al.(2021)Dujmovi\'{c}, Morin, and Yelle}]{DMY21}
	\textsc{Vida Dujmovi\'{c}, Pat Morin, and C\'{e}line Yelle}.
	\newblock \href{https://doi.org/10.7155/jgaa.00549}{Two results on layered
		pathwidth and linear layouts}.
	\newblock \emph{J. Graph Algorithms Appl.}, 25(1):43--57, 2021.
	
	\bibitem[{Dujmovi{\'c} et~al.(2004)Dujmovi{\'c}, P\'or, and Wood}]{DPW04}
	\textsc{Vida Dujmovi{\'c}, Attila P\'or, and David~R. Wood}.
	\newblock \href{http://dmtcs.episciences.org/315}{Track layouts of graphs}.
	\newblock \emph{Discrete Math. Theor. Comput. Sci.}, 6(2):497--522, 2004.
	
	\bibitem[{Dujmovi{\'c} et~al.(2016)Dujmovi{\'c}, Sidiropoulos, and
		Wood}]{DSW16}
	\textsc{Vida Dujmovi{\'c}, Anastasios Sidiropoulos, and David~R. Wood}.
	\newblock \href{https://doi.org/10.4086/cjtcs.2016.001}{Layouts of expander
		graphs}.
	\newblock \emph{Chicago J. Theoret. Comput. Sci.}, 2016(1), 2016.
	
	\bibitem[{Dujmovi{\'c} et~al.(2007{\natexlab{b}})Dujmovi{\'c}, Suderman, and
		Wood}]{DSW07}
	\textsc{Vida Dujmovi{\'c}, Matthew Suderman, and David~R. Wood}.
	\newblock \href{https://doi.org/10.1016/j.comgeo.2006.08.002}{Graph drawings
		with few slopes}.
	\newblock \emph{Comput. Geom. Theory Appl.}, 38:181--193, 2007{\natexlab{b}}.
	
	\bibitem[{Dujmovi{\'c} and Wood(2005)}]{DujWoo05}
	\textsc{Vida Dujmovi{\'c} and David~R. Wood}.
	\newblock \href{http://dmtcs.episciences.org/346}{Stacks, queues and tracks:
		Layouts of graph subdivisions}.
	\newblock \emph{Discrete Math. Theor. Comput. Sci.}, 7:155--202, 2005.
	
	\bibitem[{Dujmovi{\'c} and Wood(2007)}]{DW07}
	\textsc{Vida Dujmovi{\'c} and David~R. Wood}.
	\newblock \href{https://doi.org/10.1007/s00454-007-1318-7}{Graph treewidth and
		geometric thickness parameters}.
	\newblock \emph{Discrete Comput. Geom.}, 37(4):641--670, 2007.
	
	\bibitem[{Dujmovi{\'c} and Wood(2011)}]{DujWoo11}
	\textsc{Vida Dujmovi{\'c} and David~R. Wood}.
	\newblock \href{https://dmtcs.episciences.org/550}{On the book thickness of
		{$k$}-trees}.
	\newblock \emph{Discrete Math. Theor. Comput. Sci.}, 13(3):39--44, 2011.
	
	\bibitem[{Endo(1997)}]{Endo97}
	\textsc{Toshiki Endo}.
	\newblock \href{https://doi.org/10.1016/S0012-365X(96)00144-6}{The pagenumber
		of toroidal graphs is at most seven}.
	\newblock \emph{Discrete Math.}, 175(1--3):87--96, 1997.
	
	\bibitem[{Eppstein(2001)}]{Eppstein01}
	\textsc{David Eppstein}.
	\newblock \href{http://arxiv.org/abs/math/0109195}{Separating geometric
		thickness from book thickness}.
	\newblock 2001, arXiv:math/0109195.
	
	\bibitem[{Fedorov(1885)}]{Fedorov}
	\textsc{Evgraf~S. Fedorov}.
	\newblock \textcyr{Начала учения о фигурах} [{I}ntroduction
	to the theory of figures] (in {R}ussian).
	\newblock 1885.
	
	\bibitem[{Galil et~al.(1989)Galil, Kannan, and Szemer\'{e}di}]{GKS89}
	\textsc{Zvi Galil, Ravi Kannan, and Endre Szemer\'{e}di}.
	\newblock \href{https://doi.org/10.1007/BF02122679}{On $3$-pushdown graphs with
		large separators}.
	\newblock \emph{Combinatorica}, 9(1):9--19, 1989.
	
	\bibitem[{Ganley and Heath(2001)}]{GH01}
	\textsc{Joseph~L. Ganley and Lenwood~S. Heath}.
	\newblock \href{https://doi.org/10.1016/S0166-218X(00)00178-5}{The pagenumber
		of $k$-trees is ${O}(k)$}.
	\newblock \emph{Discrete Appl. Math.}, 109(3):215--221, 2001.
	
	\bibitem[{Gromov(2010)}]{Gromov10}
	\textsc{Mikhail Gromov}.
	\newblock \href{https://doi.org/10.1007/s00039-010-0073-8}{{S}ingularities,
		expanders and topology of maps. {P}art 2: from combinatorics to topology via
		algebraic isoperimetry}.
	\newblock \emph{Geom. Funct. Anal.}, 20(2):416–526, 2010.
	
	\bibitem[{Haslinger and Stadler(1999)}]{HS99}
	\textsc{Christian Haslinger and Peter~F. Stadler}.
	\newblock \href{https://doi.org/10.1006/bulm.1998.0085}{{RNA} structures with
		pseudo-knots: {G}raph-theoretical, combinatorial, and statistical
		properties}.
	\newblock \emph{Bull. Math. Biology}, 61(3):437--467, 1999.
	
	\bibitem[{Hatcher(2002)}]{Hatcher02}
	\textsc{Allen Hatcher}.
	\newblock Algebraic topology.
	\newblock Cambridge University Press, 2002.
	
	\bibitem[{Haxell et~al.(2003)Haxell, Szab\'o, and Tardos}]{HST03}
	\textsc{Penny Haxell, Tibor Szab\'o, and G\'abor Tardos}.
	\newblock \href{https://doi.org/10.1016/S0095-8956(03)00031-5}{Bounded size
		components---partitions and transversals}.
	\newblock \emph{J. Combin. Theory Ser. B}, 88(2):281--297, 2003.
	
	\bibitem[{Heath(1984)}]{Heath-FOCS84}
	\textsc{Lenwood~S. Heath}.
	\newblock \href{https://doi.org/10.1109/SFCS.1984.715903}{Embedding planar
		graphs in seven pages}.
	\newblock In \emph{Proc. 25th Annual Symp. Found. Comput. Sci. (FOCS '84)}, pp.
	74--83. IEEE, 1984.
	
	\bibitem[{Heath and Istrail(1992)}]{HI92}
	\textsc{Lenwood~S. Heath and Sorin Istrail}.
	\newblock \href{https://doi.org/10.1145/146637.146643}{The pagenumber of genus
		$g$ graphs is ${O}(g)$}.
	\newblock \emph{J. ACM}, 39(3):479--501, 1992.
	
	\bibitem[{Heath et~al.(1992)Heath, Leighton, and Rosenberg}]{HLR92}
	\textsc{Lenwood~S. Heath, F.~Thomson Leighton, and Arnold~L. Rosenberg}.
	\newblock \href{https://doi.org/10.1137/0405031}{Comparing queues and stacks as
		mechanisms for laying out graphs}.
	\newblock \emph{SIAM J. Discrete Math.}, 5(3):398--412, 1992.
	
	\bibitem[{Hickingbotham and Wood(2021)}]{HW21b}
	\textsc{Robert Hickingbotham and David~R. Wood}.
	\newblock \href{http://arxiv.org/abs/2111.12412}{Shallow minors, graph products
		and beyond planar graphs}.
	\newblock 2021, arXiv:2111.12412.
	
	\bibitem[{Kainen(1990)}]{Kainen90}
	\textsc{Paul~C. Kainen}.
	\newblock The book thickness of a graph. {II}.
	\newblock In \emph{Proc. 20th Southeastern Conf. Combinatorics, Graph Theory,
		and Computing}, vol.~71 of \emph{Congr. Numer.}, pp. 127--132. 1990.
	
	\bibitem[{Keszegh et~al.(2008)Keszegh, Pach, P{\'{a}}lv{\"{o}}lgyi, and
		T{\'{o}}th}]{KPPT08}
	\textsc{Bal{\'{a}}zs Keszegh, J{\'{a}}nos Pach, D{\"{o}}m{\"{o}}t{\"{o}}r
		P{\'{a}}lv{\"{o}}lgyi, and G{\'{e}}za T{\'{o}}th}.
	\newblock \href{https://doi.org/10.1016/j.comgeo.2007.05.003}{Drawing cubic
		graphs with at most five slopes}.
	\newblock \emph{Comput. Geom.}, 40(2):138--147, 2008.
	
	\bibitem[{Krauthgamer and Lee(2007)}]{KL07}
	\textsc{Robert Krauthgamer and James~R. Lee}.
	\newblock \href{https://doi.org/10.1007/s00493-007-2183-y}{The intrinsic
		dimensionality of graphs}.
	\newblock \emph{Combinatorica}, 27(5):551--585, 2007.
	
	\bibitem[{Malitz(1994{\natexlab{a}})}]{Malitz94b}
	\textsc{Seth~M. Malitz}.
	\newblock \href{https://doi.org/10.1006/jagm.1994.1028}{Genus $g$ graphs have
		pagenumber ${O}(\sqrt g)$}.
	\newblock \emph{J. Algorithms}, 17(1):85--109, 1994{\natexlab{a}}.
	
	\bibitem[{Malitz(1994{\natexlab{b}})}]{Malitz94a}
	\textsc{Seth~M. Malitz}.
	\newblock \href{https://doi.org/10.1006/jagm.1994.1027}{Graphs with ${E}$ edges
		have pagenumber ${O}(\sqrt {E})$}.
	\newblock \emph{J. Algorithms}, 17(1):71--84, 1994{\natexlab{b}}.
	
	\bibitem[{Ollmann(1973)}]{Ollmann73}
	\textsc{L.~Taylor Ollmann}.
	\newblock On the book thicknesses of various graphs.
	\newblock In \textsc{Frederick Hoffman, Roy~B. Levow, and Robert S.~D. Thomas},
	eds., \emph{Proc. 4th Southeastern Conf. Combinatorics, Graph Theory and
		Computing}, vol. VIII of \emph{Congr. Numer.}, p. 459. Utilitas Math., 1973.
	
	\bibitem[{Pach and P{\'a}lv{\"o}lgyi(2006)}]{PP06}
	\textsc{J{\'a}nos Pach and D{\"o}m{\"o}t{\"o}r P{\'a}lv{\"o}lgyi}.
	\newblock \href{https://doi.org/10.37236/1139}{Bounded-degree graphs can have
		arbitrarily large slope numbers}.
	\newblock \emph{Electron. J. Combin.}, 13(1):N1, 2006.
	
	\bibitem[{Paley(1933)}]{Paley}
	\textsc{Raymond E. A.~C. Paley}.
	\newblock \href{https://doi.org/10.1002/sapm1933121311}{On orthogonal
		matrices}.
	\newblock \emph{J. Math. Physics}, 12:311--320, 1933.
	
	\bibitem[{Pemmaraju(1992)}]{Pemmaraju92}
	\textsc{Sriram~V. Pemmaraju}.
	\newblock \href{https://www.proquest.com/docview/304023021}{Exploring the
		powers of stacks and queues via graph layouts}.
	\newblock Ph.D. thesis, Virginia Polytechnic Institute and State University,
	U.S.A., 1992.
	
	\bibitem[{Pupyrev(2020)}]{Pupyrev20a}
	\textsc{Sergey Pupyrev}.
	\newblock \href{http://arxiv.org/abs/2007.15102}{Book embeddings of graph
		products}.
	\newblock 2020, arXiv:2007.15102.
	
	\bibitem[{Rosenberg(1983)}]{Rosenberg83a}
	\textsc{Arnold~L. Rosenberg}.
	\newblock \href{https://doi.org/10.1109/TC.1983.1676134}{The {DIOGENES}
		approach to testable fault-tolerant arrays of processors}.
	\newblock \emph{IEEE Trans. Comput.}, 32(10):902--910, 1983.
	
	\bibitem[{Shahrokhi et~al.(1996)Shahrokhi, S{\'y}kora, Sz{\'e}kely, and
		V{\v{r}}{\soft{t}}o}]{SSSV-JGT96}
	\textsc{Farhad Shahrokhi, Ondrej S{\'y}kora, L{\'a}szl{\'o}~A. Sz{\'e}kely, and
		Imrich V{\v{r}}{\soft{t}}o}.
	\newblock
	\href{https://doi.org/10.1002/(SICI)1097-0118(199604)21:4<413::AID-JGT7>3.0.CO;2-S}{The
		book crossing number of a graph}.
	\newblock \emph{J. Graph Theory}, 21(4):413--424, 1996.
	
	\bibitem[{Sylvester(1867)}]{Sylvester}
	\textsc{James~Joseph Sylvester}.
	\newblock \href{https://doi.org/10.1080/14786446708639914}{{LX}. {T}houghts on
		inverse orthogonal matrices, simultaneous signsuccessions, and tessellated
		pavements in two or more colours, with applications to {N}ewton's rule,
		ornamental tile-work, and the theory of numbers}.
	\newblock \emph{The London, Edinburgh, and Dublin Philosophical Magazine and
		Journal of Science}, 34(232):461--475, 1867.
	
	\bibitem[{Vandenbussche et~al.(2009)Vandenbussche, West, and Yu}]{VWY09}
	\textsc{Jennifer Vandenbussche, Douglas~B. West, and Gexin Yu}.
	\newblock \href{https://doi.org/10.1137/080714208}{On the pagenumber of
		$k$-trees}.
	\newblock \emph{SIAM J. Discrete Math.}, 23(3):1455--1464, 2009.
	
	\bibitem[{Wood(2001)}]{Wood-GD01}
	\textsc{David~R. Wood}.
	\newblock \href{https://doi.org/10.1007/3-540-45848-4\_25}{Bounded degree book
		embeddings and three-dimensional orthogonal graph drawing}.
	\newblock In \textsc{Petra Mutzel, Michael J{\"{u}}nger, and Sebastian
		Leipert}, eds., \emph{Proc. 9th Int'l Symp. Graph Drawing (GD '01)}, vol.
	2265 of \emph{Lecture Notes in Comput. Sci.}, pp. 312--327. Springer, 2001.
	
	\bibitem[{Wood(2005)}]{Wood05}
	\textsc{David~R. Wood}.
	\newblock \href{https://doi.org/10.46298/dmtcs.352}{Queue layouts of graph
		products and powers}.
	\newblock \emph{Discrete Math. Theor. Comput. Sci.}, 7(1):255--268, 2005.
	
	\bibitem[{Wood(2018)}]{Wood18}
	\textsc{David~R. Wood}.
	\newblock \href{https://doi.org/10.37236/7406}{Defective and clustered graph
		colouring}.
	\newblock \emph{Electron. J. Combin.}, DS23, 2018.
	\newblock Version 1.
	
	\bibitem[{Yannakakis(1989)}]{Yannakakis89}
	\textsc{Mihalis Yannakakis}.
	\newblock \href{https://doi.org/10.1016/0022-0000(89)90032-9}{Embedding planar
		graphs in four pages}.
	\newblock \emph{J. Comput. System Sci.}, 38(1):36--67, 1989.
	
	\bibitem[{Yannakakis(2020)}]{Yann20}
	\textsc{Mihalis Yannakakis}.
	\newblock \href{https://doi.org/10.1016/j.jctb.2020.05.008}{Planar graphs that
		need four pages}.
	\newblock \emph{J. Combin. Theory Ser. B}, 145:241--263, 2020.
	
\end{thebibliography}
%
%
\def\soft#1{\leavevmode\setbox0=\hbox{h}\dimen7=\ht0\advance \dimen7
	by-1ex\relax\if t#1\relax\rlap{\raise.6\dimen7
		\hbox{\kern.3ex\char'47}}#1\relax\else\if T#1\relax
	\rlap{\raise.5\dimen7\hbox{\kern1.3ex\char'47}}#1\relax \else\if
	d#1\relax\rlap{\raise.5\dimen7\hbox{\kern.9ex \char'47}}#1\relax\else\if
	D#1\relax\rlap{\raise.5\dimen7 \hbox{\kern1.4ex\char'47}}#1\relax\else\if
	l#1\relax \rlap{\raise.5\dimen7\hbox{\kern.4ex\char'47}}#1\relax \else\if
	L#1\relax\rlap{\raise.5\dimen7\hbox{\kern.7ex
			\char'47}}#1\relax\else\message{accent \string\soft \space #1 not
		defined!}#1\relax\fi\fi\fi\fi\fi\fi}

}
\end{document}